\documentclass{amsart}
\usepackage{graphicx}
\usepackage{hyperref}
\usepackage{tikz}
\usepackage{comment}

\newtheorem{thm}{Theorem}[section]
\newtheorem{prop}[thm]{Proposition}
\newtheorem{lem}[thm]{Lemma}
\newtheorem{cor}[thm]{Corollary}
\newtheorem{prob}[thm]{Problem}

\theoremstyle{definition}
\newtheorem{definition}[thm]{Definition}
\newtheorem{example}[thm]{Example}

\theoremstyle{remark}

\newtheorem{remark}[thm]{Remark}

\numberwithin{equation}{section}

\newcommand{\RR}{\mathbb{R}}
\newcommand{\ZZ}{\mathbb{Z}}

\DeclareMathOperator{\cl}{\mathrm{cl}}
\DeclareMathOperator{\scl}{\mathrm{scl}}

\newcommand{\bg}{g}
\newcommand{\hg}{\hat{g}}
\newcommand{\bG}{G}
\newcommand{\hG}{\hat{G}}

\newcommand{\Symp}{\mathrm{Symp}}
\newcommand{\Ham}{\mathrm{Ham}}
\newcommand{\Supp}{\mathrm{Supp}}
\newcommand{\flux}{\mathrm{Flux}}

\begin{document}

\title{$\hat{G}$-invariant  quasimorphisms and symplectic geometry of surfaces}

\author{Morimichi Kawasaki}
\address[Morimichi Kawasaki]{Research Institute for Mathematical Sciences, Kyoto University, Kyoto 606-8502, Japan}
\email{kawasaki@kurims.kyoto-u.ac.jp}

\author{Mitsuaki Kimura}
\address[Mitsuaki Kimura]{Graduate School of Mathematical Sciences, The University of Tokyo, Tokyo, 153-8914, Japan}
\email{mkimura@ms.u-tokyo.ac.jp}

\begin{comment}
今後やるべきこと
川崎16や木村17やBKなどの部分擬凖同型の構成を引用すべきかどうか。
Bavardの双対定理の意義を擬準同型とｓｃｌの対応で説明。
小野先生の指摘を反映したものをarXivに上げた後で系1.11関係の記述はすべて削除。remark 1.12以外にも証明も忘れずに消す。
有界コホモロジー関係の記述をどうするか。
イントロの整理。
\end{comment}

\begin{abstract}
Let $\hat{G}$ be a group and $G$ its normal subgroup.
In this paper, we study $\hat{G}$-invariant quasimorphisms on $G$ which appear in symplectic geometry and low dimensional topology.
As its application, we prove the non-existence of a section of the flux homomorphism on closed surfaces of higher genus.
We also prove that Py's Calabi quasimorphism and Entov--Polterovich's partial Calabi quasimorphism are non-extendable to the group of symplectomorphisms.
We show that Py's Calabi quasimorphism is the unique non-extendable quasimorphism to some group.
%To study this concept, we introduce the notion of the $(\hG,\bG)$-commutator length and prove a Bavard-type duality theorem.
%We also study the extension problem of ${\rm Symp}_0(\Sigma_g)$-invariant (relative) quasimorphisms on ${\rm Ham}(\Sigma_g)$.
\end{abstract}

\maketitle

\tableofcontents

\section{Introduction}
In this paper, we introduce and study the notions of $\hG$-invariant quasimorphism and $(\hG,\bG)$-commutator length.
Many examples in this paper come from the symplectic geometry.
See Section \ref{symplectic section} for notions in the symplectic geometry.
\subsection{$\hG$-invariant quasimorphism}
A real-valued function $\phi$ on a group $G$ is a \emph{quasimorphism}
if there exists a constant $C$ such that
\[|\phi(gh)-\phi(g)-\phi(h)| \leq C\]
for all $g,h \in G$.
Such the smallest $C$ is called the \emph{defect} of $\phi$ and denoted by $D(\phi)$.
A quasimorphism $\phi$ on $G$ is \emph{homogeneous} if $\phi(g^n)=n\phi(g)$ for all $g \in G$ and $n \in \ZZ$.
For a group $G$, let $Q(G)$ denote the real linear space of homogeneous quasimorphisms on $G$.

The main object we consider in this paper is \emph{$\hG$-invariant quasimorphism}.

\begin{definition}
  For a group $\hG$ and its normal subgroup $\bG$,
  we say that a quasimorphism $\phi \colon \bG \to \mathbb{R}$ on $\bG$ is \emph{$\hG$-invariant} if
  $\phi(\hg \bg \hg^{-1})=\phi(\bg)$ for all $\hg \in \hG$ and $\bg \in \bG$.
  The real linear space of $\hG$-invariant homogeneous quasimorphisms on $G$ is denoted by $Q(G)^{\hG}$.
\end{definition}

Quasimorphisms appear in dynamical systems as the rotation number, in symplectic topology as spectral invariants, in geometric group theory as a characterization of non-positively curved groups, in the theory of bounded cohomology and so on (see, e.g., \cite{Ca,Fr}).
$\hG$-invariant quasimorphisms on $\bG$ also appear in several contexts. For example,
\begin{itemize}
  \item
  Let $(M,\omega)$ be a symplectic manifold.
Let $\hG$ be the identity component of the group $\Symp_0(M,\omega)$ of symplectomorphisms and
$\bG$ the group $\Ham(M,\omega)$ of Hamiltonian diffeomorphisms \cite[et al.]{EP03, GG, Py06, Bra, BKS, FOOO}.
  \item $\hG$ is the mapping class group $\mathcal{M}(\Sigma)$ of a compact oriented surface $\Sigma$ with non-empty boundary and
  $\bG$ is the Torelli group $\mathcal{I}(\Sigma)$ of $\Sigma$ or the Johnson kernel $\mathcal{K}(\Sigma)$ of $\Sigma$ \cite{CHH}.
\end{itemize}

In \cite{BM}, Brandenbursky and Marcinkowski also studied a similar concept, \emph{Aut-invariant quasimorphism} (i.e., quasimorphism which is invariant under the automorphisms), for the free group $F_n$ of rank $n$ and the surface group $\Gamma_g$ of genus $g$.
Since $F_n$ and $\Gamma_g$ have trivial center, they are isomorphic to the inner automorphism group and regarded as a normal subgroup of the automorphism group ${\rm Aut}(F_n)$ and ${\rm Aut}(\Gamma_g)$. Therefore, these Aut-invariant quasimorphisms can be seen as ${\rm Aut}(F_n)$-invariant quasimorphisms and ${\rm Aut}(\Gamma_g)$-invariant quasimorphisms in our sence.

%As widely known, the theory of quasimorphisms are related with the bounded cohomology of groups.
%the ``$\hat{G}$-invariant'' bounded cohomology plays an important role in the theory of bounded cohomology.

In the present paper, we provide some observations and applications of $\hat{G}$-invariant quasimorphisms.

%They study ${\rm Aut}(F_n)$-invariant quasimorphisms on the free group $F_n$ of rank $n$ and ${\rm Aut}(\Gamma_g)$-invariant quasimorphisms on the surface group $\Gamma_g$ of genus $g$.
% We note that we can regard $\Gamma_g$ as a normal subgroup of ${\rm Aut}(\Gamma_g)$.

\subsection{Bavard-type duality theorem}
For a group $G$, let $\cl_G$ denote the commutator length on $[G,G]$ and \emph{the stable commutator length} $\scl_G$ is defined by
$\scl_G(x)=\lim_{n \to \infty} \cl_G(x^n)/n$ for $x \in [G,G]$.
%The following \emph{Bavard duality theorem}, which relates quasimorphisms and stable commutator length ($\scl$), is one of the most fundamental results in the theory of quasimorphisms.
As written in Calegari's famous book \cite{Ca}, quasimorphism and $\scl$ have a very deep relation.
The following Bavard's duality is a symbolic theorem of that relation.
\begin{thm}[\cite{Bav}]\label{original Bavard}
  Let $G$ be a group. For any $x \in [G,G]$,
  \[ \scl_{G}(x) = \sup_{\phi \in Q(G)}  \frac12 \frac{|\phi(x)|}{D(\phi)}. \]
\end{thm}
It is a natural question what is to $\hG$-invariant quasimorphism is what $\scl$ is to quasimorphism.
To answer this question, we will show a Bavard-type duality for $\hG$-invariant quasimorphisms and a variant of commutator length.
We refer to an element of the form $[\hg,\bg]$, where $\hg \in \hG$ and $\bg \in \bG$, as a \emph{$(\hG,\bG)$-commutator}.
We define \emph{the $(\hG,\bG)$-commutator subgroup} $[\hG,\bG]$, \emph{the $(\hG,\bG)$-commutator length} $\cl_{\hG,\bG}$ and \emph{the stable $(\hG,\bG)$-commutator length} $\scl_{\hG,\bG}$ in the same way as the ordinary ones (see Section 2.1).

\begin{thm} \label{thm:bavard}
  Assume that $\bG=[\hG,\bG]$. For any $x \in [\hG,\bG]$,
  \[ \scl_{\hG,\bG}(x) = \sup_{\phi \in Q(G)^{\hG}} \frac12 \frac{|\phi(x)|}{D(\phi)}. \]
\end{thm}

We use this theorem to prove Proposition \ref{section and scl comparison}. Since $\bG$ is a normal subgroup of $\hG$, we have $[\bG,\bG]<[\hG,\bG]<\bG$.
Thus, we note that $\bG=[\hG,\bG]$ if $\bG$ is perfect \textit{i.e.} $\bG=[\bG,\bG]$.

\subsection{Comparison with the ordinary commutator length}
We say that two functions $\nu$ and $\mu$ are \textit{equivalent} if there are positive constants $C_1$ and $C_2$ such that $C_1\mu \leq \nu \leq C_2 \mu$.
In \cite{CZ}, Calegari and Zhuang gave a concept of $W$-length generalizes the commutator length.
They proved that the stabilization of some $W$-lengths are equivalent to the stable commutator length \cite[Corollary 3.25]{CZ}.
In this paper, we consider a similar problem for our situation.
Namely, we compare our norm $\cl_{\hG,\bG}$ with the norms $\cl_{\hG}$ or $\cl_{\bG}$.

We can prove that the stabilizations of $\cl_{\hG,\bG}$ and $\cl_{\hG}$
are equivalent in the following situation.

\begin{prop}\label{section and scl comparison}
  Let $\bG$ be a normal subgroup of a group $\hG$. Assume that $\bG=[\hG,\bG]$.
  If there exists a section homomorphism of the qutient map $q\colon\hG \to \hG/\bG$ \textit{i.e.} there is a group homomorphism $s \colon \hG/\bG \to \hG$ such that $q\circ s=\mathrm{id}$, then
  \[\scl_{\hG}(x)\leq\scl_{\hG,\bG}(x) \leq 2\scl_{\hG}(x)\]
for any $x \in [\hG,\bG]$.
\end{prop}

Because we use Theorem \ref{thm:bavard} to prove Proposition \ref{section and scl comparison}, the authors do not know whether $\cl_{\hG,\bG}$ and $\cl_{\hG}$ (not stabilized) are equivalent or not.

\begin{example}\label{braid and index hom}

Let $\hG$ be the braid group $B_n$ of $n$ strands and $\bG$ its commutator subgroup $[B_n,B_n]$. For any integer $n>4$, $\bG$ is a perfect group \cite{GL}, especially $\bG=[\hG,\bG]$.
It is known that $\hG/\bG \cong \ZZ$ and the abelianization map $\hG \to \hG/\bG$ is given by
the index sum homomorphism $\hG \to \ZZ$ defined by $\sigma_i \mapsto 1$ for $i=1,2,\dots,n-1$,
where $\sigma_i$ is the $i$th Artin generator. Since there is a section homomorphism $s\colon \ZZ \to \hG$,
the pair $(\hG,\bG)$ satisfies the assumptions of Proposition \ref{section and scl comparison} if $n>4$.

\end{example}

\begin{example}\label{exact and Calabi hom}
Let $(M,\omega)$ be an exact symplectic manifold.
Let $\hG$ be the group $\Ham(M,\omega)$ of Hamiltonian diffeomorphisms and $\bG$ the commutator subgroup of $\Ham(M,\omega)$.
Let $\mathrm{Cal}\colon\Ham(M,\omega)\to\RR$ denote the Calabi homomorphism (See Section \ref{symplectic section}).

It is known that $\hG/\bG \cong \RR$ and the abelianization map $\hG \to \hG/\bG$ is given by the Calabi homomorphism \cite{Ban}.
We can take a time-independent Hamiltonian function $H\colon M\to\RR$ such that $\mathrm{Cal}(H)=1$ (for instance, consider a function supported on a Darboux ball).
Then, the map $s\colon\RR\to\Ham(M,\omega)$ defined by $s(t)=\varphi_{tH}$ is a section homomorphism of $\mathrm{Cal}$.
Since it is known that $\bG$ is a perfect group (\cite{Ban}),
the pair $(\hG,\bG)$ satisfies the assumptions of Proposition \ref{section and scl comparison}.
\end{example}

\begin{example}\label{torus section}
Let $T^2$ be a 2-dimensional torus and $\omega$ a symplectic form on $T^2$ such that $\int_{T^2}\omega=1$.
Let $\hG$ be the identity component $\Symp_0(T^2,\omega)$ of the group of symplectomorphisms of $(T^2,\omega)$ and $\bG$ the group $\Ham(T^2,\omega)$ of Hamiltonian diffeomorphisms of $(T^2,\omega)$.
Let $\flux_\omega \colon \Symp_0(T^2,\omega) \to H^1(T^2 ; \RR)/H^1(T^2 ; \ZZ)$ be the (descended) flux homomorphism (See Section \ref{symplectic section}).
Then, $\mathrm{Ker}(\flux_\omega)=\bG$ and $\bG$ is known to be perfect \cite{Ban}.
Thus, since there exists a section homomorphism of $\flux_\omega \colon \Symp_0(T^2,\omega) \to H^1(T^2 ; \RR)/H^1(T^2 ; \ZZ)$,
$\hG$ and $\bG$ satisfy the assumption of Proposition \ref{section and scl comparison}.
Hence, $\scl_{\hG,\bG}$ and $\scl_{\hG}$ are equivalent.
\end{example}

%\begin{remark}
%  For $\hG=\Symp_0(\Sigma_g,\omega)$ and $\bG=\Ham(\Sigma_g,\omega)$,
%Theorem \ref{two scl different on Ham} states that $\scl_{\hG,\bG}$ and $\scl_{\hG}$ are not equivalent  if $g\geq 2$.
%On the other hand, they are equivalent if $g=1$ by Example \ref{torus section}.
%\end{remark}

However, in the following example, $\scl_{\hG,\bG}(x)$ and $\scl_{\hG}$ are not equivalent.

\begin{thm}\label{two scl different on Ham}
Let $\Sigma$ be a closed orientable surface whose genus is larger than one and $\omega$ a symplectic form on $\Sigma$.
Set $\hG=\Symp_0(\Sigma,\omega)$ and $\bG=\Ham(\Sigma,\omega)$.
Then, there exists $f\in\bG$ such that $\scl_{\hG,\bG}(f)>0$ and $\scl_{\hG}(f)=0$.
\end{thm}

By Proposition \ref{section and scl comparison}, Theorem \ref{two scl different on Ham} gives the following negative answer to a symplectic version of (Nielsen) realization problem.

\begin{cor}\label{section of flux}
Let $\Sigma$ be a closed orientable surface whose genus is larger than one and $\omega$ a symplectic form on $\Sigma$.
Then, there is no section homomorphism of the flux homomorphism $\flux_\omega\colon\Symp_0(\Sigma,\omega)\to H^1(\Sigma;\RR)$.
% \textit{i.e.} there is no homomorphism $s\colon H^1(\Sigma;\RR)\to\Symp_0(\Sigma,\omega)$ such that $\flux_\omega\circ s=\mathrm{id}$.
\end{cor}

For various versions of (Nielsen) realization problems by diffeomorphisms, \cite{MT} is a good survey.
%\begin{proof}[Proof of Corollary \ref{section of flux}]
%By Proposition \ref{section and scl comparison}
%\end{proof}

Corollary \ref{section of flux} is slightly surprising because the following proposition is essentially proved by Fathi.
\begin{prop}[\cite{Fa}]\label{Fathi}
Let $M$ be an $n$-dimensional closed manifold and $\Omega$ a volume form on $M$.
Suppose that $n\geq3$ and there is a basis of $H_1(M;\RR)$ which is represented by embedded curves having tubular neighborhoods.
Then, there is a section homomorphism of the flux homomorphism $\flux_\Omega\colon \widetilde{\mathrm{Diff}}_0(M,\Omega)\to H^{n-1}(M;\RR)$.
\end{prop}

Note that for a closed orientable surface $\Sigma$ whose genus is larger than 1 and a symplectic form $\omega$,
$\widetilde{\mathrm{Diff}}_0(\Sigma,\omega)=\widetilde{{\Symp}}_0(\Sigma,\omega)={\Symp}_0(\Sigma,\omega)$.
We also note that the symplectic flux homomorphism corresponds to the volume flux homomorphism when the dimension of the manifold is two.
Thus, Corollary \ref{section of flux} shows that Proposition \ref{Fathi} does not hold when $n=2$.

We have the following geometric interpretation of Corollary \ref{section of flux}.
For a vector field $X$ on a manifold, let $\mathcal{L}_X$ and $\iota_X$ denote the Lie derivative and the interior product with respect to $X$, respectively.

\begin{cor}\label{nonexistence of vector field}
Let $\Sigma$ be a closed orientable surface whose genus is larger than one and $\omega$ a symplectic form on $\Sigma$.
There are no smooth vector fields $X_1,\ldots,X_{2g}$ on $\Sigma$ satisfying the following conditions.
\begin{enumerate}
  \item $\mathcal{L}_{X_i}\omega=0$,
  \item $\{[\iota_{X_1}\omega],\ldots,[\iota_{X_{2g}}\omega]\}$ is a basis of $H^1(\Sigma;\RR)$,
  \item $[X_i,X_j]=0$ for any $i,j$.
\end{enumerate}
\end{cor}

\begin{remark}
Kaoru Ono pointed out that we can prove Corollary \ref{nonexistence of vector field} by an elementary calculation of vector analysis.
\end{remark}

We also provide examples of $\bG$, $\hG$ and $\alpha \in [\hG,\bG]$
such that $\scl_{\hG,\bG}(\alpha)=0$ and $\scl_{\bG}(\alpha)>0$ are not equivalent even if the quotient group $\hG/\bG$ is a finite group.
(see Proposition \ref{clhgbg and clbg}).

\subsection{Extension problem of (partial) quasimorphisms}\label{extension subsec}
It is a quite natural problem whether a homogeneous quasimorphism $\phi$ on $\bG$ can be extended as a homogeneous quasimorphism on $\hG$.
It is known that every homogeneous quasimorphism on $\hG$ is $\hG$-invariant (\cite{Ca}).
Thus, we see that $\hG$-invariance is necessary to be extended to $\phi\colon \bG\to\RR$ to a homogeneous quasimorphism on $\hG$.
Shtern and the first author also studied a similar topic \cite{Sh,Ka18}.

We give a sufficient condition of quasimorphisms to be extended.

\begin{prop}\label{prop:section_hom homog}
  Let $\bG$ be a normal subgroup of a group $\hG$.
  If there exists a section homomorphism $s \colon \hG/\bG \to \hG$ of the quotient homomorphism $\hG\to\hG/\bG$,
  then for any homogeneous $\hG$-invariant quasimorphism $\phi$ on $\bG$, there exists a homogeneous quasimorphism $\hat\phi$  on $\hG$ such that
  $\hat\phi|_{\bG}=\phi$ and $D(\hat\phi) \leq 2 D(\phi)$.
\end{prop}

On the other hand, we also give an example of non-extendable quasimorphism.

Shtern \cite[Example 1]{Sh} provided an example of $\hG$-invariant homomorphism on $\bG$ which cannot be extended to $\hG$ as a quasimorphism when $\hG$ is the Heisenberg group and $\bG$ is the commutator subgroup of $\hG$.

 For a closed orientable surface $\Sigma$ whose genus is larger than one and a symplectic form $\omega$ on $\Sigma$,
Py constructed a Calabi quasimorphism $\mu_P\colon\Ham(\Sigma,\omega)\to\RR$ called \textit{Py's Calabi quasimorphism} \cite{Py06}.
Py's Calabi quasimorphism $\mu_P$ is known to be a $\Symp_0(\Sigma,\omega)$-invariant quasimorphism.

\begin{thm}\label{Py nonex}
Let $\Sigma$ be a closed orientable surface whose genus is larger than one and $\omega$ a symplectic form on $\Sigma$.
There does not exists a homogeneous quasimorphism $\hat\mu$  on $\Symp_0(\Sigma,\omega)$ such that $\hat\mu|_{\Symp_0(\Sigma,\omega)}=\mu_P$.
\end{thm}
We note that Proposition \ref{prop:section_hom homog} and Theorem \ref{Py nonex} give another proof of Corollary \ref{section of flux}.

%As we explain in Section \ref{symplectic section}, $\mu_P$ is non-extendable to $\Symp_0(\Sigma,\omega)$.
Theorem \ref{Py nonex} has the following corollary.
To explain it, we introduce some notions.
For a closed orientable surface $\Sigma$ whose genus is larger than one, let $B_n(\Sigma)$ denote the full braid group on $n$ strings on $\Sigma$.
For a symplectic form $\omega$ on $\Sigma$, Brandenbursky \cite{Bra} constructed a liner map $\Gamma_n\colon Q(B_n(\Sigma))\to Q(\Ham(\Sigma,\omega))^{\Symp_0(\Sigma,\omega)}$ by generalizing Gambaudo-Ghys' idea \cite{GG}.

Generalizing and sophisticating Ishida's idea \cite{I}, Brandenbursky \cite{Bra} proved that the image $\mathrm{Im}(\Gamma_n)$ of $\Gamma_n$ is infinite-dimensional vector space for any $n\geq2$.
Moreover, he proved that the image $\mathrm{Im}(\Gamma_2)$ of $\Gamma_2$ contains infinitely many $\Symp_0(\Sigma,\omega)$-invariant Calabi quasimorphisms.
Thus, it is a natural problem whether Py's Calabi quasimorphism $\mu_P$ can be constructed by Brandenbursky's method or not.
We note that there exists a linear map $\bar\Gamma_n\colon Q(B_n(\Sigma))\to Q(\Symp_0(\Sigma,\omega))$ and $\Gamma_n=i_Q^\ast\circ\bar\Gamma_n$, where $i^{\ast}_Q\colon Q(\Symp_0(\Sigma,\omega))\to Q(\Ham(\Sigma,\omega))^{\Symp_0(\Sigma,\omega)}$ is the restriction map.
In particular, all elements of $\mathrm{Im}(\Gamma_n)$ are known to be extendable to $\Symp_0(\Sigma,\omega)$.
Hence, we obtain the following corollary of Theorem \ref{Py nonex}.

\begin{cor}\label{Brandenbursky Py}
Let $\Sigma$ be a closed orientable surface whose genus is larger than one and $\omega$ a symplectic form on $\Sigma$.
Then,
$\mu_P\notin\mathrm{Im}(\Gamma_n)$ for any $n\geq2$.
\end{cor}

We also consider ``$C^0$-versions'' of Theorem \ref{Py nonex} (Theorem \ref{C0 nonextend}) and Corollary \ref{section of flux} (Theorem \ref{c0 non-section}).
For this application, we use not only Py's Calabi quasimorphism, but also Brandenbursky's Calabi quasimorphism.

\begin{comment}
By combining the above two ideas, we can construct a nontrivial second cohomology class of the identity component of the group of symplectic homeomorphisms.
More precisely, we have the following theorem.
\begin{thm}\label{c0 second cohom}
Let $\Sigma$ be a closed orientable surface whose genus is larger than one and $\omega$ a symplectic form on $\Sigma$.
Then, the seconde cohomology $H^2(\mathrm{Sympeo}_0(\Sigma,\omega))$ is non-trivial.
\end{thm}

Due to Kotschick and Morita \cite{KM}, it is known that $H^2(\Symp_0(\Sigma,\omega))$ is non-trivial.
To prove that, they used the perfectness of $\Ham(\Sigma,\omega)$.
In our $C^0$-case, we have no corresponding fact and thus we guess that Theorem \ref{c0 second cohom}.
\end{comment}

We also study the extension problem of ``partial quasimorphisms''.
For the precise definitions of partial quasimorphism and its extendability, see Section \ref{partial_qm_section}.
%In this paper, we provide examples of $\hG$-invariant partial quasimorphisms on $\bG$ which cannot be extended to $\hG$ as a partial quasimorphism when $\hG$ is the identity component of the group of sympletomorphism of surfaces and $\bG$ is the commutator subgroup of $\hG$.

For a closed orientable surface $\Sigma$ and a symplectic form $\omega$ on $\Sigma$,
Entov and Polterovich constructed a partial quasimorphism $\mu_{EP}\colon\Ham(\Sigma,\omega)\to\RR$ as the asymptotization of the Oh-Schwarz spectral invariant (\cite{EP06}).
The asymptotization $\mu_{EP}$ is a semi-homogeneous $\nu_{\Ham(U)}$-quasimorphism for any displaceable open subset $U$ of $M$.
(Note that we regard $\Ham(U,\omega)$ as a subgroup of $\Ham(M,\omega)$)

\begin{thm}\label{EP nonex}
Let $\Sigma$ be a closed orientable surface of positive genus and $\omega$ a symplectic form on $\Sigma$.
Then, Entov--Polterovich's partial Calabi quasimorphism $\mu_{EP}\colon\Ham(\Sigma,\omega)\to\RR$ is non-extendable to $\Symp_0(\Sigma,\omega)$.
\end{thm}

Theorem \ref{EP nonex} is interesting because of the following reason.
As we noted in Example \ref{torus section}, the (descended) flux homomorphism $\mathrm{Flux}_\omega\colon\Symp_0(T^2,\omega)\to H^1(T^2;\RR)/H^1(T^2;\ZZ)$ has a section homomorphism.
Thus, Theorem \ref{EP nonex} shows that the same statement as Proposition \ref{prop:section_hom homog} does not hold for partial quasimorphisms.

\subsection{The space of non-extendable quasimorphisms}\label{space nonex subsec}
We study the space of non-extendable quasimorphisms.
We show that that space can be described in terms of the group cohomology under some assumptions.
For notations on the cohomology and the bounded cohomology of discrete groups, see Section \ref{2nd coh section}. 
\begin{thm} \label{thm:embed}
Let $1 \to G  \xrightarrow{i} \hG \xrightarrow{q} H \to 1$ be an exact sequence of groups.
Let $c_{\hG} \colon H_b^2(\hG;\RR) \to H^2(\hG;\RR)$ denote the comparison map.
If $H^1(G;\RR)=0$ and $H$ is amenable, there are an isomorphism
  \[\tau_{\hG} \colon Q(G)^{\hG}/Q(\hG) \to \mathrm{Im}(c_{\hG})\cap\mathrm{Im}(q^\ast)\subset H^2(\hG;\RR)\]
  and an injective homomorphism
  \[\tau_H \colon Q(G)^{\hG}/Q(\hG) \to H^2(H;\RR).\]
Here we regard $Q(\hG)$ as a subgroup of $Q(G)^{\hG}$ by the restriction map.
\end{thm}

Theorem \ref{thm:embed} is very useful when we study non-extendable quasimorphisms on $\Ham(M,\omega)$.

\begin{cor}\label{symp:embed}
  Let $(M,\omega)$ be a closed symplectic manifold.
Then, there are an isomorphism
  \[\tau_{\Symp_0} \colon Q\left(\Ham(M, \omega)\right)^{\Symp_0(M, \omega)}/Q(\Symp_0(M, \omega)) \to \mathrm{Im}(c_{\Symp_0(M,\omega)})\cap\mathrm{Im}(\flux_\omega^\ast)\]
  and an injective homomorphism
  \[\tau_{H^1} \colon Q\left(\Ham(M, \omega)\right)^{\Symp_0(M, \omega)}/Q(\Symp_0(M, \omega)) \to H^2(H^1(M; \RR)/\Gamma_\omega;\RR).\]
 Here, $\Gamma_\omega$ is the symplectic flux group.
\end{cor}

Let $\Sigma$ be a closed orientable surface whose genus is larger than one and $\omega$ a symplectic form on $\Sigma$.
As a corollary of  Corollary \ref{symp:embed}, we see that $\tau_{\hG}([\mu_P])\in H^2(\Symp_0(\Sigma,\omega))$ and $\tau_H([\mu_P])\in H^2(H^1(\Sigma;\RR))$ are non-trivial cohomology classes.
Since we do not know a precise description of the map $(i^\ast_b)^{-1}\colon H^2_b(G;\RR)^{\hG} \to H^2_b(\hG;\RR)$ used in the construction of $\tau_{\Symp_0}$ and $\tau_{H^1}$, we do not know the precise value of $\tau_{\Symp_0}([\mu_P])$ and $\tau_{H^1}([\mu_P])$.
As pointed out by Kotschick and Morita \cite{KM}, $H^\ast(H^1(\Sigma;\RR);\RR)$ can be a very large space.
However, they constructed a natural embedding $\iota_{KM}\colon H_\ast(T^{2g};\RR)\to H^\ast(H^1(\Sigma;\RR);\RR)$ and calculated $\flux^\ast\circ\iota_{KM}(H_\ast(T^{2g};\RR))$,
where $g$ is the genus of $\Sigma$.
It is an interesting problem whether $\tau_H([\mu_P])\in \iota_{KM}(H_2(T^{2g};\RR))$ or not.

We also prove that Py's Calabi quasimorphism $\mu_P$ is the ``unique'' quasimorphism which is non-extendable to some subgroup of $\Symp_0(\Sigma,\omega)$.
\begin{cor}\label{Py unique}
Let $\Sigma$ be a closed orientable surface whose genus is larger than one and $\omega$ a symplectic form on $\Sigma$.
Let $f_0,g_0$ be elements of $\Symp_0(\Sigma,\omega)$ defined in Section \ref{symplectic section} and set $a=\flux_\omega(f_0)$, $b=\flux_\omega(g_0)$ and $\hG_0=(\flux_\omega)^{-1}(\ZZ\langle a,b\rangle)$, where $\ZZ\langle a,b\rangle$ is the lattice generated by $a$ and $b$ in $H^1(\Sigma,\RR)$.
Then,
  \[Q\left(\Ham(\Sigma,\omega)\right)^{\hG_0}/Q(\hG_0)=\RR\langle[\mu_P]\rangle,\]
  where $\RR\langle[\mu_P]\rangle$ is the linear subspace spanned by $[\mu_P]$ of $Q\left(\Ham(\Sigma,\omega)\right)^{\hG_0}/Q(\hG_0)$.
\end{cor}

%As pointed out above, $H^2(H^1(M; \RR)/\Gamma_\omega;\RR)$ can be a very large space.
%However, w
When $\hG$ is a small subspace of $\Symp_0(M,\omega)$,
we show that the space of non-extendable quasimorphisms is finite-dimensional.
\begin{cor}\label{symp:dim}
  Let $(M,\omega)$ be a closed symplectic manifold and $H$ be a subgroup of $H^1(M;\RR)/\Gamma_\omega$ which is isomorphic to $\ZZ^N$ as a group.
  We set $\hG=(\flux_\omega)^{-1}(H)\subset\Symp_0(M, \omega)$.
Then,
  \[\mathrm{dim}_{\RR}\left(Q\left(\Ham(M,\omega)\right)^{\hG}/Q(\hG)\right)\leq N(N-1)/2.\]
\end{cor}

\subsection{Organization of the paper}

In Section \ref{bavard section}, we give a proof of Theorem \ref{thm:bavard}.
In Section \ref{comparison section}, we prove some results on the comparison of commutator lengths. We give a proof of Proposition \ref{prop:section_hom homog} in Subsection \ref{extend condition subsection} since we use it to prove these results.
In Section \ref{partial_qm_section}, we summarize the definitions and lemmas about partial quasimorphisms and their non-extendability.
In Section \ref{symplectic section}, we prepare notions in symplectic geometry and prove Theorem \ref{Py nonex}.
In Section \ref{C0 section}, we prove a result of non-extendability in $C^0$ setting (Theorem \ref{C0 nonextend}).
In Section \ref{2nd coh section}, we prove Theorem \ref{thm:embed} and Corollary \ref{symp:embed}, \ref{Py unique} and \ref{symp:dim}.
%As a corollary of Section \ref{symplectic section} and \ref{C0 section}, we construct a non-trivial second cohomology class on $\Symp_0(M,\omega)$ and ${\rm Sympeo}_0(M,\omega)$ in Section \ref{2nd coh section}.

%\begin{prob}[\cite{Ka18}]
%\end{prob}
%the first author gave a positive evidence of this conjecture.

%Since a quasimorphism is a $\nu_0$-quasimorphism,

\section{$\hG$-invariant Bavard duality}\label{bavard section}

\subsection{$(\hG,\bG)$-commutator length}

We recall that a $(\hG,\bG)$-commutator is an element $[\hg,\bg]$ with $\hg \in \hG$ and
$\bg \in \bG$.
Let $[\hG,\bG]$ denote the subgroup of $\bG$ generated by $(\hG,\bG)$-commutators.
For $x \in [\hG,\bG]$ we define the $(\hG,\bG)$-commutator length $\cl_{\hG,\bG}(x)$ of $x$ by
the smallest number of $(\hG,\bG)$-commutators whose product is equal to $x$.
Since $\cl_{\hG,\bG}$ is subadditive, the limit
$\scl_{\hG,\bG}(x)=\lim_{n\to\infty} \cl_{\hG,\bG}(x^n)/n$ exists.

\begin{lem}\label{lem:lower_bound}
Let $\phi$ be a $\hG$-invariant homogeneous quasimorphism on $\bG$. For any $x \in [\hG,\bG]$,
\[ \scl_{\hG,\bG} (x) \geq \frac{1}{2}\frac{|\phi(x)|}{D(\phi)}. \]

\end{lem}

\begin{proof}

Note that $|\phi([\hg,\bg])|
=|\phi([\hg,\bg]) - \phi(\hg \bg \hg^{-1}) - \phi(g^{-1}) | \leq D(\phi)$
for any $(\hG,\bG)$-commutator $[\hg,\bg] \in [\hG,\bG]$.
If $x^n$ is a product of $(\hG,\bG)$-commutators $c_1,\dots,c_m$, then we obtain an inequality
\[n|\phi(x)|=|\phi(x^n)|\leq (m-1)D(\phi) + \sum_{k=1}^k|\phi(c_k)| < 2m D(\phi).\]
and the lemma follows from it.
\end{proof}

\subsection{Proof of the duality theorem}
Now we give a proof of Theorem \ref{thm:bavard}.
For proving the equality, it is sufficient to prove the inequalities in both directions.
One side follows from Lemma \ref{lem:lower_bound}, thus we prove the other side (Proposition \ref{prop:upper_bound}).
For this purpose, we use the strategy in Calegari--Zhuang's work \cite{CZ} (see also \cite{Ka17}).
Some parts of the proof go through in the same way as the arguments in \cite{Ka17}.
Moreover, some parts are much easier than the ones in \cite{Ka17} because a technical lemma corresponding to \cite[Lemma 2.6]{Ka17} follows immediately in our situation.
Thus, we often omit such parts of the proof.

Set $\Gamma=[\hG,\bG]$ and define a set
\[A_{\Gamma}= \bigsqcup_{k=0}^{\infty}(\Gamma \times \mathbb{R})^k.\]
Let $x_1^{s_1}\cdots x_k^{s_k}$ denote elements of $A_{\Gamma}$,
where $x_1,\ldots, x_k \in \Gamma$ and $s_1,\ldots,s_k \in \mathbb{R}$.
We define a function $\| \cdot \|_{\Gamma} \colon A_{\Gamma} \to \mathbb{R}_{\geq 0}$ by
\[ \| x_1^{s_1}\cdots x_k^{n_k} \|_{\Gamma}= \lim_{n\to \infty} \frac{1}{n} \cl_{\hG,\bG}(x_1^{\lfloor s_1 n \rfloor}\cdots x_k^{\lfloor s_k n \rfloor}), \]
where $\lfloor t \rfloor$ is the integer part of $t \in \mathbb{R}$.
The function $\| \cdot \|_{\Gamma} \colon A_{\Gamma} \to \mathbb{R}_{\geq 0}$ is well-defined
\cite[Proposition 2.1]{Ka17}.

We define some operations on $A_{\Gamma}$.
For elements $\mathsf{ x}=x_1^{s_1}\dots x_k^{s_k}$, $\mathsf{ y}=y_1^{t_1}\dots y_l^{t_l}$ of $A_{\Gamma}$ and a real number $\lambda$,
we define $\mathsf{ x} \star \mathsf{ y}$, $\bar{\mathsf{ x}}$, and $\mathsf{ x}^{(\lambda)}$ by
\[  \mathsf{  x} \star \mathsf{ y}  = x_1^{s_1} \dots x_k^{s_k} y_1^{t_1}\dots y_{l}^{t_l}, \quad   \bar{\mathsf{ x}}  = x_k^{-s_k}\dots x_1^{-s_1},
  \quad {\rm and} \quad   \mathsf{ x}^{(\lambda)} = x_1^{\lambda s_1}\dots x_k^{\lambda s_k}. \]
We define the equivalence relation $ \sim$ on $A_{\Gamma}$ by
$\mathsf{ x} \sim \mathsf{ y}$ if and only if $\| \mathsf{ x}\bar{\mathsf{ y}} \|_{\Gamma} =0$
for $\mathsf{x}, \mathsf{y} \in A_{\Gamma}$.
Let $A$ denote the quotient set $A_{\Gamma} / \sim $.
The function $\| \cdot \|_{\Gamma} \colon A_{\Gamma} \to \mathbb{R}_{\geq 0}$ on $A_{\Gamma}$ induces the function
$\| \cdot \|\colon A \to \mathbb{R}_{\geq 0}$ on $A$.
Let $[\mathsf{ x}] \in A$ denote the equivalence class of  $\mathsf{ x} \in A_{\Gamma}$.
For $\boldsymbol{x}=[\mathsf{ x}]$, $\boldsymbol{y}=[\mathsf{ y}]$ in $A$ and a real number $\lambda$, we define
$\boldsymbol{x}+\boldsymbol{y}$ and $\lambda \boldsymbol{x}$ by
\[  \boldsymbol{x}+\boldsymbol{y}  = [\mathsf{ x} \star \mathsf{y}] \quad {\rm and} \quad   \lambda \boldsymbol{x}  = [ \mathsf{ x}^{(\lambda)}].\]
These operators are well-defined \cite[Proposition 2.2]{Ka17}
 and $(A, \| \cdot \|)$ is a normed vector space \cite[Proposition 2.3]{Ka17}.
By the Hahn-Banach theorem, we obtain the following proposition.

\begin{prop} \label{prop:hahn-banach}
For any $\boldsymbol{x} \in A$,
\[ \| \boldsymbol{x} \| = \sup_{\tilde\phi \in A^*} \frac{\tilde\phi( \boldsymbol{x})}{\| \tilde\phi \|^{\ast}}, \]

where $A^*$ is the dual space of $A$ and $\| \cdot \|^{\ast}$ is the dual norm on $A^*$.
\end{prop}

On the other hand, we can construct a $\hG$-invariant quasimorphism from an element of $A^*$ in the following way.
\begin{prop} \label{prop:G-inv_hqm}
  For $\tilde\phi \in A^*$, the function $\phi \colon \Gamma \to \mathbb{R}$ defined by $\phi(x)=\tilde\phi([x^1])$ is a $\hG$-invariant homogeneous quasimorphism.
  Moreover, $D(\phi) \leq \frac12 \| \tilde\phi  \|^*$.
\end{prop}

\begin{proof}
First, we prove that $\phi$ is a quasimorphism. For any $x, y \in \Gamma$,

\begin{align*}
  & \quad |\phi(xy) - \phi(x) -\phi (y)  | \\
  &= |\tilde\phi([(xy)^1]) - \tilde\phi([x^1]) - \tilde\phi([y^1])| \\
  &= |\tilde\phi([(xy)^1]+(-1)[x^1]+(-1)[y^1])| \\
  & \leq  \| \tilde\phi  \|^* \| (xy)^1 \star x^{-1} \star y^{-1} \|_{\Gamma}   \\
  &=\| \tilde\phi  \|^* \cdot \lim_{n\to \infty} \frac{1}{n}\cl_{\hG,\bG}((xy)^n x^{-n} y^{-n}). \\
\end{align*}

Since $(xy)^{2n}x^{-2n}y^{-2n}$ is a product of $n$ commutators
(see \cite[Lemma 2.24]{Ca} for example),
\[ \lim_{n\to \infty} \frac{1}{n}\cl_{\hG,\bG}((xy)^n x^{-n} y^{-n})=\lim_{n\to \infty} \frac{1}{2n}\cl_{\hG,\bG}((xy)^{2n} x^{-2n} y^{-2n})\leq \frac12.\]
Hence
\[|\phi(xy) - \phi(x) -\phi (y)  | \leq \frac12 \| \tilde\phi  \|^*.\]
Therefore, $\phi$ is a quasimorphism and $D(\phi) \leq \frac12 \| \tilde\phi  \|^*$.

Next, we prove that $\phi$ is homogenous.
Since $(x^n)^1 \sim x^n$ for any $x \in [\hG,\bG]$ and any integer $n$,
\[\phi(x^n)
=\tilde\phi([(x^n)^1])
=\tilde\phi([x^n])
=\tilde\phi(n[x^1])\]
for any $x \in \Gamma$ and any integer $n$.
Since $\tilde\phi \colon A \to \mathbb{R}$ is a linear map,
\[\tilde\phi(n[x^1])
=n\tilde\phi([x^1])
=n \phi(x).\]
for any $x \in \Gamma$ and any integer $n$.
Hence $\phi$ is homogeneous.

Finally, we prove that $\phi$ is $G$-invariant.
For any $\hg \in \hG$ and any $x \in \Gamma\subset G$,

\begin{align*}
  & \quad |\phi(\hg x \hg^{-1}) - \phi(x) | \\
  & = |\tilde\phi([(\hg x \hg^{-1})^1]) - \tilde\phi([x^1])| \\
  & = |\tilde\phi([(\hg x \hg^{-1})^1] +(-1) [x^1])| \\
  & \leq \| \tilde\phi  \|^* \| ( \hg x \hg^{-1})^1 \star x^{-1} \|_{\Gamma}   \\
  & = \| \tilde\phi  \|^* \cdot \lim_{n\to \infty} \frac{1}{n}\cl_{\hG,\bG}((\hg x \hg^{-1})^n  x^{-n}) \\
  & = \| \tilde\phi  \|^* \cdot \lim_{n\to \infty} \frac{1}{n}\cl_{\hG,\bG}([\hg,x^n]) \\
  & = 0.
\end{align*}

Therefore, $\phi$ is $\hG$-invariant.
We complete the proof.
\end{proof}
As a corollary of Propositions \ref{prop:hahn-banach} and \ref{prop:G-inv_hqm}, we have the following proposition.
Theorem \ref{thm:bavard} follows from this proposition and Lemma \ref{lem:lower_bound}.

\begin{prop}\label{prop:upper_bound}
  For any $x \in [\hG,\bG]$,
  \[ \scl_{\hG,\bG}(x) \leq \sup_{\phi \in Q([\hG, \bG])^{\hG}}  \frac12\frac{|\phi(x)|}{D(\phi)}. \]
\end{prop}

\begin{proof}
  By Proposition \ref{prop:hahn-banach} and \ref{prop:G-inv_hqm}, since $D(\phi) \leq \frac12 \| \phi \|^*$,
  \[ \scl_{\hG,\bG} (x) = \| x^1 \|= \sup_{\tilde\phi \in A^*} \frac{\tilde\phi( [x^1] )}{\| \tilde\phi \|^*}
\leq \sup_{\phi } \frac{1}{2}\frac{\phi(x)}{D(\phi)}. \qedhere\]
\end{proof}

\section{Comparison of commutator lengths}\label{comparison section}

We compare the $(\hG, \bG)$-commutator length $\cl_{\hG,\bG}$ with
the ordinary commutator lengths $\cl_{\hG}$ of $\hG$ and $\cl_{\bG}$ of $\bG$.
By definition,
$\cl_{\hG} \leq \cl_{\hG,\bG}$
on $[\hG,\bG]$, and
$ \cl_{\hG,\bG} \leq \cl_{\bG}$
on $[\bG,\bG]$.

\subsection{A condition of quasimorphisms to be extended}\label{extend condition subsection}

First, we give a proof of Proposition \ref{prop:section_hom homog}.
It also follows from the result of Shtern \cite[Theorem 3]{Sh}.
However, we provide an estimate of the defect in order to prove Proposition \ref{section and scl comparison}.

\begin{proof}[Proof of Proposition \ref{prop:section_hom homog}]
  Let $\pi \colon \hG \to \hG/\bG $ be the natural projection.
  For $\hg \in \hG$, we set $q_{\hg} = s(\pi(\hg))$ and $\bg_{\hg} = q_{\hg}^{-1}  \hg   \in \bG$ . We define the function
  $\phi^\prime \colon \hG \to \mathbb{R}$
by
$\phi^\prime(\hg)=\phi(\bg_{\hg})$.
Since $s \circ \pi$ is a homomorphism, $q_{\hg_1 \hg_2}= q_{\hg_1} q_{\hg_2}$
for $\hg_1, \hg_2 \in \hG$.
Thus

\begin{align*}
  & |\phi^\prime(\hg_1 \hg_2)-\phi^\prime(\hg_1)-\phi^\prime(\hg_2)| \\
 &=  |\phi(\bg_{\hg_1 \hg_2})-\phi(\bg_{\hg_1})-\phi(\bg_{\hg_2})| \\
 &=  |\phi(q_{\hg_2}^{-1}q_{\hg_1}^{-1} \hg_1 \hg_2)-\phi(q_{\hg_1}^{-1}\hg_1)-\phi(q_{\hg_2}^{-1} \hg_2)| \\
 &=  |\phi(q_{\hg_1}^{-1} \hg_1 \hg_2 q_{\hg_2}^{-1})-\phi(q_{\hg_1}^{-1}\hg_1)-\phi(\hg_2 q_{\hg_2}^{-1})| \\
 & \leq D(\phi).
\end{align*}
Hence, $\phi^\prime$ is a quasimorphism with $D(\phi^\prime)\leq D(\phi)$.
Define the function $\hat\phi \colon \hG \to \RR$, which is called the \emph{homogenization} of $\phi^\prime$, by
$\hat\phi(\hg)=\lim_{n\to \infty} \phi(\hg^n)/n$ for $\hg \in \hG$.
It is known that $\hat\phi$ a homogeneous quasimorphism and
$D(\hat\phi) \leq 2 D(\phi^\prime)$ (\cite{Ca}, Corollary 2.59).
By $\pi\circ s$ is the identity map, $\phi^\prime$ is an extension of $\phi$ to $\hG$.
Since $\phi$ is a homogeneous quasimorphism, $\hat\phi$ is also an extension of $\phi$ to $\hG$.
Hence, we complete the proof.
\end{proof}

\subsection{$\cl_{\hG,\bG}$ vs $\cl_{\hG}$}

Now we prove Proposition \ref{section and scl comparison} which states that
$\scl_{\hG,\bG}$ and $\scl_{\hG}$ are equivalent if there exists a section homomorphism.

\begin{proof}[Proof of Proposition \ref{section and scl comparison}]
The inequality $\scl_{\hG}(x)\leq\scl_{\hG,\bG}(x)$ immediately follows from the definitions of norms.
Thus, we prove $\scl_{\hG,\bG}(x)\leq 2\scl_{\hG}(x)$ below.

  By Theorem \ref{thm:bavard}, for any $\epsilon>0$, there exists a $\hG$-invariant homogeneous quasimorphism $\phi$ such that
  \[\scl_{\hG,\bG} (x) -\epsilon \leq \frac{1}{2}\frac{\phi(x)}{D(\phi)}. \]
By Proposition \ref{prop:section_hom homog}, there exists an extension
$\hat\phi$ of $\phi$ which is homogeneous and $D(\hat\phi) \leq 2 D(\phi^\prime)$ .
Therefore,

  \[\frac{1}{2}\frac{\phi(x)}{D(\phi)} \leq \frac{\hat\phi(x)}{D(\hat\phi)}\leq 2\scl_{\hG}(x).\]
Since $\epsilon$ can be taken arbitrary small, we have finished the proof.
\end{proof}

\subsection{$\cl_{\hG,\bG}$ vs $\cl_{\bG}$}
We give an example of a pair $(\hG,\bG)$ of groups such that
$\scl_{\hG,\bG}$ and $\scl_{\bG}$ are not equivalent even if the quotient group $\hG/\bG$ is a finite group.

Let $B_3$ and $P_3$ denote the braid group and the pure braid group on 3 strands, respectively.
Set $\Delta=\sigma_1 \sigma_2 \sigma_1 =\sigma_2 \sigma_1 \sigma_2 $, where $\sigma_1$ and $\sigma_2$ are the Artin generators.
Note that $\Delta^2$ is the full twist.
Set $x=\sigma_1^2$, $y=\sigma_2^2$ and $z=\Delta^2$.
Then $P_3$ has a presentation
\[ P_3=\langle x,y,z \; | \; xz=zx,yz=zy  \rangle \cong F_2 \times \mathbb{Z}.\]

\begin{prop}\label{clhgbg and clbg}
For $\hG=B_3$ and $\bG=P_3$, there exists an element $\alpha \in [\bG,\bG]$
such that $\scl_{\hG,\bG}(\alpha)=0$ and $\scl_{\bG}(\alpha)>0$.
\end{prop}

To prove Proposition \ref{clhgbg and clbg}, we use Brooks' \emph{counting quasimorphism} on free groups \cite{Br}.
Let $F_2=\langle x,y \rangle$ be a free group of rank 2
and $w$ a reduced word in $\{x ^{\pm1},y^{\pm1} \}$.
A \emph{counting function} $c_w \colon F_2 \to \ZZ$ is defined as $c_w (g)$ being the maximal number of disjoint copies of  $w$ in the reduced representative of $g\in F_2$.
A \emph{counting quasimorphism} is a function of the form
\[ h_w(g)=c_w (g)-c_{w^{-1}}(g). \]

\begin{proof}[Proof of Proposition \ref{clhgbg and clbg}]

We set $\alpha=[x,y]=[\sigma_1^2, \sigma_2^2]$.
Since $\Delta \alpha \Delta^{-1}= [\sigma_2^2,\sigma_1^2]=\alpha^{-1}$,
$\phi(\alpha)$ is equal to zero for every $\hG$-invariant homogeneous quasimorphism $\phi$ on $[\hG,\bG]$.
Thus, by Proposition \ref{prop:upper_bound}, $\scl_{\hG,\bG}(\alpha)=0$.

On the other hand, we can prove that $\scl_{\bG}(\alpha)>0$ as follows.
Set $\phi= \bar{h}_w \circ {\rm pr}_1$, where $w=xyx^{-1}y^{-1}$
and ${\rm pr}_1 \colon  P_3 \cong F_2 \times \mathbb{Z} \to F_2$ is the first projection homomorphism.
Since $c_w([x,y]^n)=n$ and $c_{w^{-1}}([x,y]^n)=0$,
\[\bar{\phi}(\alpha)=\bar{h}_w([x,y])=1.\]
%(it also says that $\bar{\phi}$ is not a homomorphism).
Therefore, by Theorem \ref{original Bavard},
\[\scl_{\bG} (\alpha) \geq \frac{1}{2}\frac{1}{D(\bar\phi)}>0.\qedhere\]

\end{proof}

\section{Non-extendability of partial quasimorphisms}\label{partial_qm_section}

We prepare some notions on partial quasimorphisms.
Burago, Ivanov and Polterovich defined the notion of conjugation-invariant norm.

\begin{definition}[{\cite{BIP}}]\label{conjugation-invariant norm}
Let $G$ be a group. A function $\nu\colon G\to \mathbb{R}$ is called a \textit{conjugation-invariant norm} on $G$ if $\nu$ satisfies the following axioms:
\begin{enumerate}
  \item $\nu(1)=0$;
  \item $\nu(f)=\nu(f^{-1})$ for every $f\in G$;
  \item $\nu(fg)\leq \nu(f)+\nu(g)$ for every $f,g\in G$;
  \item $\nu(f)=\nu(gfg^{-1})$ for every $f,g\in G$;
  \item $\nu(f)>0$ for every $f\neq 1\in G$.
\end{enumerate}
\end{definition}

\begin{example}\label{trivial norm}
We define a function $\nu_0\colon G\to \mathbb{R}$ by
\begin{equation*}
\nu_0(g)=
\begin{cases}
0 & (g=1), \\
1 & (\text{otherwise}).
\end{cases}
\end{equation*}
Then, $\nu_0$ is a conjugation-invariant norm.
\end{example}

\begin{example}\label{fragmentation norm}
Let $G$ be a group and $H$ a subgroup of $G$.
We define the fragmentation norm $\nu_H$ with respect to $H$ by for an element $f$ of $G$,
 \[\nu_H(f)=\min\{k;\exists g_1\ldots,g_k\in G, \exists h_1,\ldots h_k\in H\text{ such that }f=g_1h_1g_1^{-1}\cdots g_kh_kg_k^{-1}\}.\]
If there is no such decomposition of $f$, we set $\nu_H(f)=+\infty$.
If $\nu(f)<+\infty$ for any $f\in G$, $\nu$ is a conjugation-invariant norm.
\end{example}

In \cite{EP06}, Entov and Polterovich essentially considered a concept of partial quasimorphism (relative quasimorphism, norm-controlled quasimorphism).

\begin{definition}\label{qm rt nu}
Let $G$ be a group and $\nu$ a conjugation-invariant norm on $G$.
A function $\phi\colon G\to\mathbb{R}$ is called a \textit{$\nu$-quasimorphism} (\textit{quasimorphism relative to $\nu$} or \textit{quasimorphism controlled by $\nu$}) if there exists a positive number $C$ such that for any elements $f$, $g$ of $G$,
\[|\phi(fg)-\phi(f)-\phi(g)|<C\,\min\{\nu(f),\nu(g)\}.\]
%The infimum of such $C$ is called \textit{the defect of} $\phi$ and let $D(\phi)$ denote the defect of $\phi$.
$\phi$ is called \textit{semi-homogeneous} if $\phi(f^n)=n\phi(f)$ for any element $f$ of $G$ and any non-negative integer $n$.
\end{definition}

We note that any quasimorphism is a $\nu_0$-quasimorphism.

\begin{definition}\label{def of extendability}
Let $G$ be a normal subgroup of a group $\hat{G}$ and $\nu\colon G\to\RR$ a conjugation-invariant norm on $G$.
A semi-homogeneous $\nu$-quasimorphism $\mu$ on $G$ is called \textit{extendable to $\hG$} if there are a conjugation-invariant norm $\hat{\nu}$ on $\hat{G}$ and a semi-homogeneous $\hat{\nu}$-quasimorphism $\hat{\mu}$ on $\hat{G}$ such that $\hat{\mu}(g)=\mu(g)$ for any $g\in G$.
A homogeneous quasimorphism $\mu$ on $G$ is called \textit{non-extendable to $\hG$} otherwise.
\end{definition}

We provide a convenient lemma for proving non-extendability.

\begin{lem}\label{general nonextendability of qm}
Let $\mu$ be a semi-homogeneous $\hG$-invariant $\nu$-quasimorphism on $G$.
Let $f$, $g$ be elements of $\hG$ satisfying
\begin{itemize}
\item
$f(gf^{-1}g^{-1})=(gf^{-1}g^{-1})f$,
\item
$[f,g]\in \bG$,
\item
$\mu([f,g])\neq0$.
\end{itemize}
Then,
$\mu$ is non-extendable to $\hG$.
\end{lem}

Lemma \ref{general nonextendability of qm} immediately follows from the following lemma.

%In order to prove Proposition \ref{general nonextendability of qm}, we prepare the following lemma.

\begin{lem}\label{extendbility implies vanishing}
Let $\nu$ be a conjugation-invariant norm on a group $\hat{G}$, $\hat{\mu}$ a semi-homogeneous $\nu$-quasimorphism on a group $\hat{G}$ and
$f$, $g$ elements of $\hat{G}$ satisfying
$(fgf^{-1})g^{-1}=g^{-1}(fgf^{-1})$.
Then, $\hat{\mu}([f,g])=0$.
\end{lem}

To prove Lemma \ref{extendbility implies vanishing}, we use the following lemma essentially proved in \cite[Theorem 1.3]{MVZ} and \cite[Lemma 3.17]{KO19}.

\begin{lem}\label{conjinv of qm}
Let $\nu$ be a conjugation-invariant norm on a group $\hat{G}$, $\hat{\mu}$ a semi-homogeneous $\nu$-quasimorphism on a group $\hat{G}$.
Then,
$\hat\mu(gfg^{-1})=\hat\mu(f)$.
\end{lem}

\begin{proof}
By the definitions of partial quasimorphism and conjugation-invariant norm, for any positive integer $k$,
\begin{align*}
      \hat\mu(f^k)&\leq\hat\mu(g)+\hat\mu(g^{-1}f^kg)+\hat\mu(g^{-1})+C\cdot\nu(g)+C\cdot\nu(g^{-1}), \\
      \hat\mu(g^{-1}f^kg)&\leq\hat\mu(g^{-1})+\hat\mu(f^k)+\mu(g)+C\cdot\nu(g^{-1})+C\cdot\nu(g).
\end{align*}
Thus,
\begin{align*}
      &\hat\mu(f^k)-\hat\mu(g)-\hat\mu(g^{-1})-C\cdot\nu(g)-C\cdot\nu(g^{-1}) \\
      &\leq\hat\mu(g^{-1}f^kg)\leq\mu(f^k)-\hat\mu(g)-\hat\mu(g^{-1})+C\cdot\nu(g)+C\cdot\nu(g^{-1})
\end{align*}
Since $\hat\mu$ is semi-homogeneous, $\hat\mu(f^k)=k\hat\mu(f)$ and $\hat\mu(g^{-1}f^kg)=\hat\mu((g^{-1}fg)^k)=k\hat\mu(g^{-1}fg)$ for any positive integer $k$.
Therefore, by dividing the above inequality by $k$ and passing to the limit as $k\to\infty$, we obtain $\hat\mu(gfg^{-1})=\hat\mu(f)$.
\end{proof}

\begin{proof}[Proof of Lemma \ref{extendbility implies vanishing}]
By $f(gf^{-1}g^{-1})=(gf^{-1}g^{-1})f$, for any integer $n$.,
\begin{equation}\label{commu lem}
[f,g]^n=(f(gf^{-1}g^{-1}))^n=f^n(gf^{-1}g^{-1})^n=f^n(gf^{-n}g^{-1})=[f^n,g].
\end{equation}
Thus, since $\hat{\mu}$ is semi-homogeneous, for any positive integer $n$,
\[n\hat\mu([f,g])=
\hat{\mu}([f,g]^n)
      =\hat{\mu}([f^n,g])
      =\hat{\mu}(f^ngf^{-n}g).
\]

Thus, by Lemma \ref{conjinv of qm}, for any integer $n$,
\begin{align*}
      &-C\cdot\nu(g) \\
      &=\hat{\mu}(g)-\hat{\mu}(g)-C\cdot\nu(g) \\
      & = \hat{\mu}(f^ngf^{-n})-\hat{\mu}(g)-C\cdot\nu(g) \\
      &\leq\hat{\mu}(f^ngf^{-n}g^{-1}) \\
      & \leq \hat{\mu}(f^ngf^{-n})+\hat{\mu}(g^{-1})+C\cdot\nu(g) \\
      &=\hat{\mu}(g)+\hat{\mu}(g^{-1})+C\cdot\nu(g).
\end{align*}

Set
\[R=\max\{|\hat{\mu}(g)+\hat{\mu}(g^{-1})+C\cdot\nu(g)|,|C\cdot\nu(g)|\}.\]
Then, by
$n\hat\mu([f,g]) =\hat{\mu}(f^ngf^{-n}g^{-1})$,
       $|\hat{\mu}([f,g])|=|\mu(f^ngf^{-n}g^{-1})|/n<R/n$ for any positive integer $n$.
Hence, $\hat{\mu}([f,g])=0$.
\end{proof}

\section{Applications to symplectic geometry}\label{symplectic section}
First, we prepare notions in symplectic geometry and the flux homomorphism.
For a more precise description, refer to \cite{Ban97,MS,P01} for example.

Let $(M,\omega)$ be a symplectic manifold.
Let $\Symp(M,\omega)$ denote the group of symplectomorphism with compact support and $\Symp_0(M,\omega)$ denote the identity component of $\Symp(M,\omega)$.
Here, we consider the $C^\infty$-topology on $\Symp(M,\omega)$.

For a Hamiltonian function $H\colon M\to\mathbb{R}$ with compact support, we define the \textit{Hamiltonian vector field} $X_H$ associated with $H$ by
\[\omega(X_H,V)=-dH(V)\text{ for any }V \in \mathcal{X}(M),\]
where $\mathcal{X}(M)$ is the set of smooth vector fields on $M$.

Let $S^1$ denote $\mathbb{R}/\mathbb{Z}$.
For a (time-dependent) Hamiltonian function $H\colon S^1\times M\to\mathbb{R}$ with compact support and for $t \in S^1$, we define a function $H_t\colon M\to\mathbb{R}$ by $H_t(x)=H(t,x)$.
Let $X_H^t$ denote the Hamiltonian vector field associated with $H_t$ by  and let $\{\varphi_H^t\}_{t\in\mathbb{R}}$ denote the isotopy generated by $X_H^t$ such that $\varphi^0=\mathrm{id}$.
Let $\varphi_H$ denote $\varphi_H^1$ and $\varphi_H$ is called \textit{the Hamiltonian diffeomorphism generated by } $H$.
For a symplectic manifold $(M,\omega)$, we define the group of Hamiltonian diffeomorphisms by
\[\Ham(M,\omega)=\{\varphi\in\mathrm{Diff}(M)\;|\;\exists H\in C^\infty(S^1\times M)\text{ such that }\varphi=\varphi_H\}.\]

We note that $\Ham(M,\omega)$ is a normal subgroup of $\Symp_0(M,\omega)$.
%$\Ham\vartriangleleft\Symp_0$

Let $X$ be a subset of a symplectic manifold $(M,\omega)$.
$X$ is \textit{displaceable} if there exists a Hamiltonian function $H\colon S^1\times M\to\mathbb{R}$ such that $\varphi_H(X)\cap\bar{X}=\emptyset$, where $\bar{X}$ is the topological closure of $X$.

For an exact symplectic manifold $(M,\omega)$, we recall that the \textit{Calabi homomorphism}
is a function $\mathrm{Cal}_M\colon\Ham(M,\omega)\to\mathbb{R}$ defined by
\[
	\mathrm{Cal}_M(\varphi_F)=\int_0^1\int_M F_t\omega^n\,dt.
\]
The Calabi homomorphism is known to be well-defined and a group homomorphism (see \cite{Cala,Ban,Ban97,MS}).

\begin{definition}\label{definition of Calabi property}
Let $\mu\colon\Ham(M,\omega)\to\mathbb{R}$ be a homogeneous quasimorphism.
An open subset $U$ of $M$ has the Calabi property with respect to $\mu$ if $\omega|_U$ is exact and the restriction of $\mu$ to $\Ham(U,\omega)$ coincides with the Calabi homomorphism $\mathrm{Cal}_U$.
\end{definition}
In terms of subadditive invariants, the Calabi property corresponds to the asymptotically vanishing spectrum condition in \cite[Definition 3.5]{KO19}

\begin{definition}[\cite{EP03,PR}]\label{definition of Calabi qm}
A \textit{Calabi quasimorphism} is a homogeneous quasimorphism $\mu\colon\Ham(M,\omega)\to\mathbb{R}$ such that any displaceable open subset of $M$ has the Calabi property with respect to $U$.
\end{definition}

% For a closed orientable surface $\Sigma_g$ whose genus is larger than one and a symplectic form $\omega$ on $\Sigma_g$,
%Py constructed a Calabi quasimorphism $\mu_P\colon\Ham(\Sigma_g,\omega)\to\RR$ called \textit{Py's Calabi quasimorphism} \cite{Py06}.
%Py's Calabi quasimorphism $\mu_P$ is known to be a $\Symp(\Sigma,\omega)$-invariant quasimorphism.

%\begin{definition}[\cite{Bannn}]\label{The flux homomorphism}
%(the definition of flux homomorphism, in the process of writing)

Here, we introduce the notion of the (volume) flux homomorphism.
Our explanation is rough.
For a more precise description, refer to \cite[Section 3]{Ban97} for example.

Let $M$ be an $n$-dimensional manifold and $\Omega$ a volume form on $M$.
Let $\mathrm{Diff}(M,\Omega)$ denote the group of diffeomorphisms preserving $\Omega$ with compact support,
$\mathrm{Diff}_0(M,\Omega)$ denote the identity component of $\mathrm{Diff}(M,\Omega)$ and
$\widetilde{\mathrm{Diff}}(M,\Omega)$ denote the universal covering of $\mathrm{Diff}(M,\Omega)$.
We define the (volume) flux homomorphism $\flux_\Omega\colon\widetilde{\mathrm{Diff}}(M,\Omega)\to H_c^{n-1}(M;\RR)$ by
\[\flux_\Omega([\{\psi^t\}_{t\in[0,1]}])=\int_0^1[\iota_{X_t}\Omega]dt,\]
where $\{\psi^t\}_{t\in[0,1]}$ is a path in $\mathrm{Diff}_0(M,\Omega)$ with $\psi^0=1$ and $[\{\psi^t\}_{t\in[0,1]}]$ is the element of the universal covering $\widetilde{\mathrm{Diff}}(M,\Omega)$ represented by the path $\{\psi^t\}_{t\in[0,1]}$.
It is known that $\flux_\Omega$ is a well-defined homomorphism.

We also define the (descended) flux homomorphism.
We set $\Gamma_\Omega=\flux_\Omega(\pi_1(\mathrm{Diff}(M,\Omega))$ which is called the \textit{volume flux group}.
Then, we naturally obtain the homomorphism $\flux_\Omega\colon\mathrm{Diff}(M,\Omega)\to H_c^{n-1}(M;\RR)/\Gamma_\Omega$

%\end{definition}

If $(M,\omega)$ is a symplectic manifold, then we can define the symplectic flux group $\Gamma_\omega$ and the symplectic flux homomorphism $\flux_\omega\colon\Symp_0(M,\omega)\to H_c^{1}(M;\RR)/\Gamma_\omega$ similarly and it is known that $\mathrm{Ker}(\flux_\omega)=\Ham(M,\omega)$ \cite{Ban,Ban97}.

%\section{Proof of Theorem \ref{two scl different on Ham} and Corollary \ref{EP nonex}}\label{symplectic application}

Let $\Sigma$ be a closed orientable surface of positive genus and $\omega$ a symplectic form on $\Sigma$.
In order to prove Theorems \ref{two scl different on Ham}, \ref{Py nonex} and \ref{EP nonex}, we prepare the following elements of $\Symp_0(\Sigma,\omega)$.

Since the genus of $\Sigma$ is positive, we can take a non-separating simple closed curve $C$ in $\Sigma$.
Then, there are a positive number $r$ and a symplectic embedding $\iota\colon(-1,1)\times \RR/r\ZZ\to\Sigma$ such that $\iota(\{0\}\times \RR/r\ZZ)=C$.
Here, the symplectic form on $(-1,1)\times \RR/r\ZZ$ is defined by $dx\wedge dy$, where $(x,y)$ is the coordinate on $(-1,1)\times \RR/r\ZZ$.

Let $\epsilon\in(0,1)$ and $\chi\colon (-1,1)\to[0,1]$ be a function satisfying the following conditions.
\begin{itemize}
\item
$\chi(x)=0$ for any $x\in(-1,-1+\epsilon)\cup(1-\epsilon,1)$,
\item
$\chi(x)+\chi(1+x)=1$ for any $x\in(-1,0)$.
\end{itemize}
By the above conditions, we see that $\chi(x)=1$ for any $x\in(-\epsilon,\epsilon)$.
Define a function $F\colon \Sigma\to\RR$ by
\begin{equation*}
F(z)=
\begin{cases}
\chi(x) & (\text{if }z=\iota(x,y)\text{ for some }(x,y)\in(-1,1)\times \RR/r\ZZ), \\
0 & (\text{if }z\notin\mathrm{Im}(\iota)).
\end{cases}
\end{equation*}
Since $C$ is non-separating, $\Sigma\setminus\mathrm{Im}(\iota)$ is path-connected.
Thus, there exists $g_0\in\Symp_0(\Sigma,\omega)$ such that $g_0(\iota(x,y))=\iota(x+1,y)$ for any $(x,y)\in(-1,0)\times \RR/r\ZZ$.

Define a map $f_0\colon\Sigma\to\Sigma$ by
\begin{equation*}
f_0(z)=
\begin{cases}
\varphi_F(z) & (\text{if }z\in\iota((-1,0)\times \RR/r\ZZ)), \\
z & (\text{otherwise}).
\end{cases}
\end{equation*}
Since $f_0(z)=z$ for any $z\in\iota((-1,-1+\epsilon)\cup(-\epsilon,\epsilon))\times \RR/r\ZZ)$, $f_0$ is well-defined and $f_0\in\Symp_0(\Sigma,\omega)$.
Since $\chi(x)+\chi(1+x)=1$ for any $x\in(-1,0)$, by the definition of $g_0$,
\begin{equation*}
g_0f_0^{-1}g_0^{-1}(z)=
\begin{cases}
\varphi_F(z) & (\text{if }z\in\iota((0,1)\times \RR/r\ZZ)), \\
z & (\text{otherwise}).
\end{cases}
\end{equation*}
Thus, we obtain $\varphi_F=f_0g_0f_0^{-1}g_0^{-1}$.
Since $\Supp(f_0)\subset\iota((-1,0)\times \RR/r\ZZ)$ and $\Supp(g_0f_0^{-1}g_0^{-1})\subset\iota((0,1)\times \RR/r\ZZ)$,
% and $\iota$ is injective,
$f_0(g_0f_0^{-1}g_0^{-1})=(g_0f_0^{-1}g_0^{-1})f_0$.

%\begin{prop}\label{Py and fg}

%Let $\Sigma$ be a connected closed surface with positive genus and $\omega$ a symplectic structure on $\Sigma$.

%Then, there exist $g_0,f_0\in\Symp_0(\Sigma,\omega)$ such that
%\[(g_0f_0g_0^{-1})f_0^{-1}=f_0^{-1}(g_0f_0g_0^{-1}),\]
%\[[g_0,f_0]\in\Ham(\Sigma,\omega)\]
%\end{prop}
To prove Proposition \ref{Py calabi annulus} and Theorem \ref{Py nonex}, we use the following properties of Py's Calabi quasimorphism.

\begin{prop}[\cite{Py06}]\label{Py calabi annulus}
Let $\Sigma$ be a closed orientable surface whose genus is larger than one, $\omega$ a symplectic form on $\Sigma$
and $U$ an open subset of $\Sigma$ which is homeomorphic to an annulus.
Then $U$ has the Calabi property with respect to Py's Calabi quasimorphism $\mu_P$.
\end{prop}

\begin{proof}[Proof of Theorem \ref{two scl different on Ham}]
By the definition of $F$, $\int_{\Sigma}F \omega>0$.
By Proposition \ref{Py calabi annulus}, $\mathrm{Im}(\iota)$ has the Calabi property with respect to $\mu_P$.
Since $\varphi_F=f_0g_0f_0^{-1}g_0^{-1}$ and $\mathrm{Supp}(F)\subset\mathrm{Im}(\iota)$,
\[\mu_P([f_0,g_0])=\mu_P(\varphi_F)=\int_{\Sigma}F \omega>0.\]
Thus, by Theorem \ref{thm:bavard},
$\scl_{\hG,\bG}([f_0,g_0])>0$.

On the other hand, by $f_0(g_0f_0^{-1}g_0^{-1})=(g_0f_0^{-1}g_0^{-1})f_0$ and a similar calculation as \eqref{commu lem}, $[f_0,g_0]^n=[f_0^n,g_0]$ for any integer $n$.
%\[[f_0,g_0]^n=(f_0(g_0f_0^{-1}g_0^{-1}))^n=f_0^n(g_0f_0^{-1}g_0^{-1})^n=f_0^n(g_0f_0^{-n}g_0^{-1})=[f_0^n,g_0].\]
Thus,
\[\mathrm{cl}_{\hG}([f_0,g_0]^n)=\mathrm{cl}_{\hG}([f_0^n,g_0])\leq1\]
 for any integer $n$.
Hence, $\scl_{\hG}([f_0,g_0])=0$.
\end{proof}

\begin{proof}[Proof of Corollary \ref{nonexistence of vector field}]
To prove by contradiction, we suppose there exist vector fields $X_1,\ldots,X_{2g}$ satisfying the conditions.

Let $\varphi_i^t$ denote the time-$t$ map of the flow generated by $X_i$.
Set $\alpha_i=[\iota_{X_i}\omega]\in H^1(\Sigma;\RR)$ for $i=1,\ldots,2g$.
Define a map $s\colon H^1(\Sigma;\RR)\to\Symp_0(\Sigma;\RR)$ by
\[s(t_1\alpha_1+t_2\alpha_2+\cdots+t_{2g}\alpha_{2g})=\varphi_1^{t_1}\circ\varphi_2^{t_2}\circ\cdots\circ\varphi_{2g}^{t_{2g}}.\]
Since $\mathcal{L}_{X_i}\omega=0$ for any $i$, $\varphi_1^{t_1}\circ\varphi_2^{t_2}\circ\cdots\circ\varphi_{2g}^{t_{2g}}\in\Symp_0(\Sigma;\RR)$.
Since $\alpha_i$ is a basis, $s$ is well-defined.
Since $[X_i,X_j]=0$ for any $i,j$, $s$ is a homomorphism.
By the definition of the flux homomorphism, $s$ is a section.
It contradicts Corollary \ref{section of flux}.
\end{proof}

%\begin{cor}\label{Branden calabi}
%Brandenbursky's Calabi quasimorphism has no Calabi property on annulus.
%\end{cor}

\begin{proof}[Proof of Theorem \ref{Py nonex}]
As we showed in the proof of Theorem \ref{two scl different on Ham},
$\mu_P([f_0,g_0])>0$.
Thus, by Lemma \ref{general nonextendability of qm}, $\mu_P$ is non-extendable to $\hG$.
Since any quasimorphism is a $\nu_0$-quasimorphism, we complete the proof of Theorem \ref{Py nonex}.
\end{proof}

To prove Theorem \ref{EP nonex}, we introduce the following property of Entov--Polterovich's partial Calabi quasimorphism $\mu_{EP}$.
This is a corollary of ``heaviness'' of $C$ in the sense of \cite{EP09}.

\begin{prop}[{\cite[Example 1.18]{EP09}}]\label{meridian heavy}
For the above Hamiltonian function $F\colon\Sigma\to\RR$,
\[\mu_{EP}(\varphi_F)=\int_\Sigma F\omega-\int_\Sigma\omega.\]
\end{prop}

\begin{proof}[Proof of Theorem \ref{EP nonex}]
Set $\hG=\Symp_0(\Sigma,\omega)$ and $\bG=\Ham(\Sigma,\omega)$.
By Proposition \ref{meridian heavy},
%To prove by contradiction, we suppose there exist a conjugation-invariant norm $\hat\nu$ on $\hG$ and a semi-homogenous $\hat\nu$-quasimorphism $\hat{\mu}_{EP}\colon \hG\to\RR$  such that $\hat{\mu}_{EP}|_{\bG}=\mu_{EP}$.
%By Lemma \ref{extendbility implies vanishing},
%$\hat\mu([g_0,f_0])=0$.

%Since $[g_0,f_0]=\varphi_G\in\Ham(\Sigma,\omega)$,
%\[\hat{\mu}([g_0,f_0])=\hat\mu(\varphi_G)=
\[\mu_{EP}(\varphi_F)=\int_\Sigma F\omega-\int_\Sigma\omega<0.\]
By $[f_0,g_0]=\varphi_F$ and Lemma \ref{general nonextendability of qm}, $\mu_{EP}$ is non-extendable to $\hG$.
\end{proof}

\section{On $C^0$-symplectic topology}\label{C0 section}

As a generalization of the (volume) flux homomorphism, Fathi \cite{Fa} considered the \textit{mass flow homomorphism} for measure-preserving homeomorphisms.
Here, for simplicity, we explain only a special case.

Let $\Sigma$ be a closed orientable surface whose genus is larger than one and $\omega$ a symplectic form on $\Sigma$.
Let $\mathrm{Sympeo}(\Sigma,\omega)$ denote the $C^0$-closure of $\Symp(\Sigma,\omega)$ in the group of homeomorphisms and
$\mathrm{Sympeo}_0(\Sigma,\omega)$ denote the identity component of $\mathrm{Sympeo}(\Sigma,\omega)$.

Then, there is a homomorphism $\theta_\omega\colon\mathrm{Sympeo}_0(\Sigma,\omega)\to H^1(\Sigma;\RR)$ called the \textit{mass flow homomorphism} of $(\Sigma,\omega)$.
We note that
\[\mathrm{Ker}(\theta_\omega)=\overline{\Ham(\Sigma,\omega)}^{C^0},\]
where $\overline{\Ham(\Sigma,\omega)}^{C^0}$ is the $C^0$-closure of the group of $\Ham(\Sigma,\omega)$ in the group of homeomorphisms \cite{CGHS}

\begin{remark}
In Fathi's original paper, the domain of the mass flow homomorphism looks different from the above one.
Oh and M$\ddot{\mathrm{u}}$ller proved that his original domain and our domain correspond \cite{OM}.
%Theorem 5.1 of OM
\end{remark}

We have the following theorem which is a $C^0$-version of Corollary \ref{section of flux}.

\begin{cor}\label{c0 non-section}
Let $\Sigma$ be a closed orientable surface whose genus is larger than one and $\omega$ a symplectic form on $\Sigma$.
There is no section homomorphism of the mass flow homomorphism $\theta_\omega\colon\mathrm{Sympeo}_0(\Sigma,\omega)\to H^1(\Sigma;\RR)$.
\end{cor}

To prove Corollary \ref{c0 non-section}, we use both of Py's Calabi quasimorphism and Brandenbursky's Calabi quasimorphism.

In order to prove Corollary \ref{c0 non-section}, we construct a non-extendable quasimorphism and apply Proposition \ref{prop:section_hom homog}.
To construct a non-extendable quasimorphism, we use the following theorem by Entov, Polterovich, and Py.
%To construct a $\hG$-invariant quasimorphism on $\bG$, we use the following.

\begin{thm}[\cite{EPP}]\label{mfh property}
Let $\Sigma$ be a closed orientable surface and $\omega$ a symplectic form on $\Sigma$.
Let $\mu$ be a homogeneous quasimorphism on $\Ham(\Sigma,\omega)$.
Suppose the following conditions.
\begin{enumerate}
  \item There exists a positive number $A$ such that for any disc $D\subset \Sigma$ of area less than $A$, the restriction of $\mu$ to $\Ham(D,\omega)$ vanishes.
  \item The restriction of $\mu$ to each one-parameter subgroup of $\Ham(\Sigma,\omega)$ is linear.
\end{enumerate}
Then, there exists a homogeneous quasimorphism $\bar{\mu}$ on $\overline{\Ham(\Sigma,\omega)}^{C^0}$ such that $\bar{\mu}|_{\Ham(\Sigma,\omega)}=\mu$.
\end{thm}

As we explained in Subsection \ref{extension subsec},
Brandenbursky proved that there are infinitely many Calabi quasimorphisms in $\mathrm{Im}(\Gamma_2)$.
We take one of them and set $\mu_B\in Q(\Ham(\Sigma,\omega))^{\Symp_0(\Sigma,\omega)}$.
Then,
\begin{prop}
The quasimorphism $\mu_P-\mu_B$ satisfies the conditions of Theorem \ref{mfh property}.
Therefore, there exists a homogeneous quasimorphism $\mu_{PB}$ on $\overline{\Ham(\Sigma,\omega)}^{C^0}$ such that $\mu_{PB}|_{\Ham(\Sigma,\omega)}=\mu_P-\mu_B$.
\end{prop}

\begin{proof}
Note that any disc $D\subset \Sigma$ of area less than $\frac{1}{2}\int_\Sigma\omega$ is displaceable.
Thus, since $\mu_P$ and $\mu_B$ are Calabi quasimorphisms, $\mu_P-\mu_B$ satisfies the first condition of Theorem \ref{mfh property}.

For a Hamiltonian function $H\colon \Sigma\to\RR$, let $f^P_H,f^B_H\colon\RR\to\RR$ be functions defined by
\[f^P_H(t)=\mu_P(\varphi_H^t), f^B_H=\mu_B(\varphi_H^t).\]
Rosenberg \cite[Theorem 8.6]{R} and Brandenbursky \cite[Theorem 2.12]{Bra} implicitly proved that $f^P_H$ and $f^B_H$ are continuous functions, respectively.
Thus, since $\mu_P$ and $\mu_B$ are homogeneous quasimorphisms, $\mu_P-\mu_B$ satisfies the second condition of Theorem \ref{mfh property}.
\end{proof}

\begin{thm}\label{C0 nonextend}
The homogeneous quasimorphism $\mu_{PB}\colon\overline{\Ham(\Sigma,\omega)}^{C^0}\to\RR$ is non-extendable to $\mathrm{Sympeo}_0(\Sigma,\omega)$.
\end{thm}

\begin{proof}
Since $[f_0,g_0]\in\Ham(\Sigma,\omega)$,
\[\mu_{PB}([f_0,g_0])=(\mu_P-\mu_B)([f_0,g_0])=\mu_P([f_0,g_0])-\mu_B([f_0,g_0]).\]
As we showed in the proof of Theorem \ref{two scl different on Ham}, $\mu_P([f_0,g_0])=\int_{\Sigma}G \omega>0$.
As explained in Subsection \ref{extension subsec}, each element of $\mathrm{Im}(\Gamma_2)$ is extendable to $\mathrm{Symp}_0(\Sigma,\omega)$.
Thus, by Lemma \ref{extendbility implies vanishing}, $\mu_B([f_0,g_0])=0$.
Hence, $\mu_{PB}([f_0,g_0])>0$.
Therefore,  by Lemma \ref{general nonextendability of qm}, $\mu_P$ is non-extendable to $\mathrm{Sympeo}_0(\Sigma,\omega)$.
\end{proof}
Since $\mathrm{Ker}(\theta_\omega)=\overline{\Ham(\Sigma,\omega)}^{C^0}$,
Corollary \ref{c0 non-section} immediately follows from Theorem \ref{C0 nonextend} and Proposition \ref{prop:section_hom homog}.

\section{The space of non-extendable quasimorphisms}\label{2nd coh section}
%We consider the cohomology of discrete groups.
For a group $G$, let $H^{\bullet}_b(G; \RR)$ denote the bounded cohomology of $G$. 
Let $\delta_G \colon Q(G) \to H_b^2(G;\RR)$ denote the differential and $c_G \colon H_b^2(G;\RR) \to H^2(G;\RR)$ the comparison map.
For a group $\hG$ and its normal subgroup $G$, let $i \colon G \to \hG$ denote the inclusion. The map $i$ induces maps $i^\ast \colon H^2(\hG;\RR) \to H^2(G;\RR)$,
$i_Q^\ast \colon Q(\hG) \to Q(G)$ and $i_b^\ast \colon H_b^2(\hG;\RR)\to H_b^2(G;\RR)$.
For more precise descriptions on the bounded cohomology, see \cite{Ca,Fr}.

To prove Theorem \ref{thm:embed}, we use the following well-known facts, see  \cite[Section 2.4]{Ca} for example.

\begin{lem}\label{lem:exact}
Let $G$ be a group.
Then, there is an exact sequence
\[ 0 \to H^1(G;\RR) \to Q(G) \xrightarrow{\delta_G} H_b^2(G;\RR) \xrightarrow{c_G} H^2(G;\RR).\]
\end{lem}

\begin{thm}[\cite{Gr}]\label{lem:amenable}
  Let $1 \to G  \xrightarrow{i} \hG \xrightarrow{q} H \to 1$ be an exact sequence of groups.
If $H$ is amenable,
then the natural homomorphisms $i_b^\ast\colon H^\bullet_b(\hG;\RR) \to H^\bullet_b(G;\RR)^{\hG}$ are isomorphisms in each dimension.
\end{thm}

Here $H^\bullet_b(G;\RR)^{\hG}$ is the invariant part of $H^\bullet_b(G;\RR)$ under the action of $\hG$ on $G$ by outer automorphisms.
Similarly, $H^\bullet (G;\RR)^{\hG}$ is the $\hG$-invariant part of $H^\bullet (G;\RR)$.

By considering a
$\hG$-invariant cochain complex $C^\bullet(G;\RR)^{\hG}$ instead of $C^\bullet(G;\RR)$,
we can obtain a $\hG$-invariant version of Lemma \ref{lem:exact}.

\begin{lem}\label{lem:G^inv_exact}
Let $\hG$ be a group and $G$ its normal subgroup.
Then, there is an exact sequence
\[ 0 \to H^1(G;\RR)^{\hG} \to Q(G)^{\hG} \xrightarrow{\delta_G} H_b^2(G;\RR)^{\hG} \xrightarrow{c_G} H^2(G;\RR)^{\hG}.\]
\end{lem}

%Moreover, the image of $i_b^\ast$ and $\delta_G$ are contained in $H_b^2(G;\RR)^{\hG/G}$.
%If $\hG/G$ is amenable, $i_b^\ast \colon H_b^2(\hG;\RR)\to H_b^2(G;\RR)^{\hG/G}$ is an isomorphism  by Theorem \ref{lem:amenable}.

\begin{proof}[Proof of Theorem \ref{thm:embed}]
  We consider the following commutative diagram. By Lemma \ref{lem:exact} and \ref{lem:G^inv_exact}, each horizontal sequences are exact.

  \begin{tikzpicture}[auto]
  \node (11) at (2, 1.5) {$Q(\hG)$}; \node (21) at (5, 1.5) {$H_b^2(\hG;\RR)$};
  \node (31) at (8, 1.5) {$H^2(\hG;\RR)$};
  \node (00) at (0, 0) {$H^1(G;\RR)^{\hG}$}; \node (10) at (2, 0) {$Q(G)^{\hG}$};
  \node (20) at (5, 0) {$H^2_b(G;\RR)^{\hG}$};  \node (30) at (8, 0) {$H^2(G;\RR)^{\hG}$};
  \node (32) at (8, 3) {$H^2(H;\RR)$};
  \draw[->] (11) to node {$\delta_{\hG}$} (21);
  \draw[->] (21) to node {$c_{\hG}$} (31);
  \draw[->] (20) to node {$c_{G}$} (30);
  \draw[->] (00) to node {} (10);
  \draw[->] (10) to node {$\delta_{G}$} (20);
  \draw[->] (11) to node {$i_Q^\ast$} (10);
  \draw[->] (21) to node {$i_b^\ast$} (20);
  \draw[->] (31) to node {$i^\ast$} (30);
  \draw[->] (32) to node {$q^\ast$} (31);
  \end{tikzpicture}

Since $H$ is amenable, by Theorem \ref{lem:amenable}, there is the inverse $(i_b^\ast)^{-1}\colon H^2(G;\RR)^{\hG} \to H^2(\hG;\RR)$ of $i_b^\ast \colon H^2(\hG;\RR) \to H^2(G;\RR)^{\hG}$.
Thus, define a map $\alpha \colon Q(G)^{\hG} \to H^2(\hG;\RR)$ by
\[\alpha(\phi) = c_{\hG} \circ (i_b^\ast)^{-1} \circ \delta_G(\phi).\]
for $\phi \in Q(G)^{\hG}$.

To construct the map $\tau_{\hG} \colon Q(G)^{\hG}/Q(\hG) \to \mathrm{Im}(c_{\hG})\cap\mathrm{Im}(q^\ast)$ from $\alpha$,
it is sufficient to prove $\alpha\circ i_Q^\ast(Q(\hG))=\{0\}$ and $\mathrm{Im}(\alpha)\subset\mathrm{Im}(c_{\hG})\cap\mathrm{Im}(q^\ast)$.
For any $\hat\phi \in Q(\hG)$, $\alpha(i^\ast_Q(\hat\phi))= c_{\hG} \circ \delta_{\hG}(\hat\phi)=0$ and hence  $\alpha\circ i_Q^\ast(Q(\hG))=\{0\}$.
Since $(i_b^\ast)^{-1}\circ\delta_G(\phi)\in H_b^2(\hG;\RR)$ for any $\phi \in Q(G)^{\hG}$, $\mathrm{Im}(\alpha)\subset\mathrm{Im}(c_{\hG})$.

To prove $\mathrm{Im}(\alpha)\subset\mathrm{Im}(q^\ast)$, we use the following exact sequence.
  The exact sequence $1 \to G \xrightarrow{i} \hG \xrightarrow{q} H \to 1$ gives rise to the 7-term exact sequence
  \begin{equation*}
  \begin{split}
    0 \to H^1(H;\RR) \to H^1(\hG;\RR) \to H^1(G;\RR)^{\hG} \to H^2(H;\RR) \\
    \xrightarrow{q^\ast} {\rm Ker}(i^\ast)   \to H^1(H;H^1(G;\RR)) \to H^3(H;\RR)
  \end{split}
  \end{equation*}
(see \cite{DHW} for example).  Since $H^1(G;\RR)=0$, the map
  \[q^\ast \colon H^2(H;\RR) \to {\rm Ker}(i^\ast)\]
   is an isomorphism.
      Since
   \[i^\ast \circ \alpha (\phi)= c_{G} \circ \delta_{G}(\phi)=0\]
  for $\phi \in Q(G)^{\hG}$, ${\rm Im} (\alpha) \subset {\rm Ker}(i^\ast)=\mathrm{Im}(q^\ast)$.
 Hence we prove $\mathrm{Im}(\alpha)\subset\mathrm{Im}(c_{\hG})\cap\mathrm{Im}(q^\ast)$ and $\alpha$ induces the map $\tau_{\hG} \colon Q(G)^{\hG}/Q(\hG) \to \mathrm{Im}(c_{\hG})\cap\mathrm{Im}(q^\ast)$.

Now, we prove that $\tau_{\hG}$ is injective.
Take $\phi\in Q(G)^{\hG}$ such that $\tau_{\hG}([\phi])=\alpha(\phi)=0$.
Then, $(i_b^\ast)^{-1} \circ \delta_G(\phi)\in\rm{Ker}(c_{\hG})=\rm{Im}(\delta_{\hG})$.
Thus, there exists $\hat\phi \in Q(\hG)$ such that
\[(i_b^\ast)^{-1} \circ \delta_G(\phi)=\delta_{\hG}(\hat\phi).\]
Thus
\[\delta_G(\phi)=i_b^\ast \circ\delta_{\hG}(\hat\phi)= \delta_{G} \circ i^\ast_Q (\hat\phi).\]
Since $H^1(G;\RR)=0$, $\delta_{G}$ is injective and hence $\phi = i^\ast_Q (\hat\phi)$.
Therefore, $\tau_{\hG}$ is injective.

  Now, we prove that $\tau_{\hG}$ is surjective.
  Take $y\in\mathrm{Im}(c_{\hG})\cap\mathrm{Im}(q^\ast)$.
  Since $y\in\mathrm{Im}(c_{\hG})$, there is $x\in H_b^2(\hG;\RR)$ such that $c_{\hG}(x)=y$.
  Since $y\in\mathrm{Im}(q^\ast)=\mathrm{Ker}(i^\ast)$, $i^\ast(y)=0$.
  Thus,
  \[c_G\circ i_b^\ast(x)=i^\ast\circ c_{\hG}(x)=i^\ast(y)=0.\]
  Hence, there is $\phi\in Q(G)^{\hG}$ such that $\delta_G(\phi)=i_b^\ast(x)$.
  Then, by $\delta_G(\phi)=i_b^\ast(x)$ and $c_{\hG}(x)=y$,
  \[\tau_{\hG}([\phi])=\alpha(\phi)=c_{\hG} \circ (i_b^\ast)^{-1} \circ \delta_G(\phi)=c_{\hG} \circ (i_b^\ast)^{-1} \circ i_b^\ast(x)=c_{\hG}(x)=y.\]
  Therefore, $\tau_{\hG}$ is surjective.

    Since $q^\ast \colon H^2(H;\RR) \to {\rm Ker}(i^\ast)=\mathrm{Im}(q^\ast)$ is an isomorphism, we define $\tau_H$ as $\tau_H = {(q^\ast)}^{-1} \circ \tau_{\hG}\colon Q(G)^{\hG}/Q(\hG) \to H^2(H;\RR)$.
Since $\tau_{\hG}$ is injective,  $\tau_H$ is also injective.
\end{proof}

\begin{proof}[Proof of Corollary \ref{symp:embed}]
Recall that every commutative group is amenable  \cite{vN} (see also \cite{Ca,Fr}).
Since $H^1(M; \RR)/\Gamma_\omega$ is a commutative group, $H^1(M; \RR)/\Gamma_\omega$ is amenable.
By Banyaga's theorem \cite{Ban, Ban97}, $\Ham(M, \omega)$ is known to be perfect, in  particular, $H^1\left(\Ham(M, \omega);\RR\right)=0$.
Thus, Corollary \ref{symp:embed} follows from Theorem \ref{thm:embed} and the exact sequence
\[1 \to \Ham(M, \omega)  \xrightarrow{i} \Symp_0(M, \omega) \xrightarrow{\flux_\omega} H^1(M; \RR)/\Gamma_\omega \to 1.\]
\end{proof}

\begin{proof}[Proof of Corollary \ref{symp:dim}]
We use the following exact sequence
\[1 \to \Ham(M, \omega)  \xrightarrow{i} \hG \xrightarrow{\flux_\omega|_{\hG}} H \to 1.\]
By a similar argument as the proof of Corollary \ref{symp:embed},  $H^1\left(\Ham(M, \omega);\RR\right)=0$ and $H$ is amenable.
Thus, by Theorem \ref{thm:embed}, there is an injective homomorphism from $Q\left(\Ham(M,\omega)\right)^{\hG}/Q(\hG)$ to $H^2(H;\RR)$.
Since $H$ is isomorphic to $\ZZ^N$,
\[\mathrm{dim}_{\RR}\left(H^2(H;\RR)\right)=\mathrm{dim}_{\RR}\left(H^2(\ZZ^N;\RR)\right)=N(N-1)/2.\]
Thus,
  \[\mathrm{dim}_{\RR}\left(Q\left(\Ham(M,\omega)\right)^{\hG}/Q(\hG)\right)\leq\mathrm{dim}_{\RR}\left(H^2(H;\RR)\right)= N(N-1)/2.\]
\end{proof}

\begin{proof}[Proof of Corollary \ref{Py unique}]
We use the following exact sequence
\[1 \to \Ham(\Sigma, \omega)  \xrightarrow{i} \hG_0 \xrightarrow{\flux_\omega|_{\hG_0}} \ZZ\langle a,b\rangle \to 1.\]
%By a similar argument as the proof of Corollary \ref{symp:embed}, $H^1\left(\Ham(M, \omega);\RR\right)=0$ and $\ZZ\langle a,b\rangle$ is amenable.
%Thus, by Theorem \ref{thm:embed}, there is an injective homomorphism from $Q\left(\Ham(M,\omega)\right)^{\hG_0}/Q(\hG_0)$ to $H^2(\ZZ\langle a,b\rangle;\RR)$.
%Since $\ZZ\langle a,b\rangle$ is isomorphic to $\ZZ^2$,
%\[\mathrm{dim}_{\RR}\left(H^2(\ZZ\langle a,b\rangle;\RR)\right)=\mathrm{dim}_{\RR}\left(H^2(\ZZ^2;\RR)\right)=1.\]
%Thus,
%  \[\mathrm{dim}_{\RR}\left(Q\left(\Ham(M,\omega)\right)^{\hG_0}/Q(\hG_0)\right)\leq\mathrm{dim}_{\RR}\left(H^2(\ZZ\langle a,b\rangle;\RR)\right)=1.\]
Since $\ZZ\langle a,b\rangle$ is isomorphic to $\ZZ^2$, by Corollary \ref{symp:dim},
\[\mathrm{dim}_{\RR}\left(Q\left(\Ham(M,\omega)\right)^{\hG_0}/Q(\hG_0)\right)\leq2(2-1)/2=1.\]
Since $f_0,g_0\in\hG_0$, by a similar argument as the proof of Theorem \ref{Py nonex}, we see that $[\mu_P]$ is a non-trivial element of $Q\left(\Ham(\Sigma,\omega)\right)^{\hG_0}/Q(\hG_0)$.
Hence, we complete the proof of Corollary \ref{Py unique}.
\end{proof}

\section*{Acknowledgment}
The authors would like to thank Shuhei Maruyama for telling them the 7-term exact sequence and Takahiro Matsushita for fruitful discussions on Proposition \ref{prop:section_hom homog} and \ref{clhgbg and clbg}.
They also thank Kazuhiko Fukui, Tomohiko Ishida, Jarek K\c{e}dra, Atsuhide Mori, Ryuma Orita, Kaoru Ono and Leonid Polterovich  for nice advices and comments.
This work has been supported by JSPS KAKENHI Grant Number JP18J00765.


\begin{thebibliography}{9}

%\bibitem[AK]{AK} V. I. ~Arnold, B. A. ~Khesin,
%	\textit{Topological methods in hydrodynamics},
%	Applied Mathematical Sciences 125, Springer, New York (1998)

\bibitem[Ban78]{Ban} A.~Banyaga,
	\textit{Sur la structure du groupe des diff\'eomorphismes qui pr\'eservent une forme symplectique},
	Comment.\ Math.\ Helv.\ \textbf{53} (1978), no.~2, 174--227.

\bibitem[Ban97]{Ban97} A.~Banyaga,
	\textit{The structure of classical diffeomorphism groups},
        Mathematics and its Applications, 400. Kluwer Academic Publishers Group, Dordrecht,  (1997).

  \bibitem[Bav]{Bav} C.~Bavard,
  \textit{Longueur stable des commutateurs},
  Enseign. Math., \textbf{37} (1991), 109--150.

%\bibitem[BHW]{BHW} Jonathan Bowden, Sebastian Hensel and Richard Webb,
%\textit{Quasi-morphisms on surface diffeomorphism groups},
%\texttt{arxiv:1909.07164}

\bibitem[Bra]{Bra} M.~Brandenbursky,
	\textit{Bi-invariant metrics and quasi-morphisms on groups of Hamiltonian diffeomorphisms of surfaces},
	Internat.\ J. Math. \textbf{26} (2015) no.~9 1550066.

%  \bibitem[BK]{BK} M.~Brandenbursky, J.~K\c{e}dra,  \emph{Fragmentation norm and relative quasimorphisms},  to appear in Proc. Am. Math. Soc.

%  \bibitem[BK]{BK} M.~Brandenbursky, J.~K\c{e}dra,  \emph

\bibitem[BKS]{BKS} M.~Brandenbursky, J.~K\c{e}dra and E.~Shelukhin,
	\textit{On the autonomous norm on the group of Hamiltonian diffeomorphisms of the torus},
	Commun.\ Contemp.\ Math.\ \textbf{20} (2018), no.~2, 1750042.

\bibitem[BM]{BM}   M. Brandenbursky and M. Marcinkowski,
\textit{Aut-invariant norms and Aut-invariant quasimorphisms on free and surface group},
Comment.\ Math.\ Helv.\, \textbf{94} (2019), 661-–687.

  \bibitem[Bro]{Br} R.~Brooks,
  \textit{Some remarks on bounded cohomology},
  Ann. of Math. Stud., \textbf{97} (1981) 53--63.

  \bibitem[BIP]{BIP} D.~Burago, S.~Ivanov and L.~Polterovich, \textit{Conjugation-invariant norms on groups of geometric origin},
  Adv. Stud. Pure. Math., \textbf{52} (2008), 221--250.

  \bibitem[Cala]{Cala} E.~Calabi,
\textit{On the group of automorphisms of a symplectic manifold}, Problem in Analysis (Lectures at the Sympos. in Honor of Salomon Bochner, Princeton University, Princeton, NJ, 1969), Princeton University Press, Princeton, NJ, 1970, 1–-26.

  \bibitem[Cale]{Ca} D.~Calegari,
  \textit{scl}, MSJ Memoirs20, Math. Soc. Japan, Tokyo, (2009).

  \bibitem[CZ]{CZ} D.~Calegari and D.~Zhuang, \textit{Stable $W$-lengths}, Contemp. Math, \textbf{560} (2011), 145--169.



\bibitem[CGHS]{CGHS}
D.~Cristofaro-Gardiner, V.~Humili\`ere and S.~Seyfaddini,
\textit{Proof of the simplicity conjecture},
arXiv:2001.01792v1.


  \bibitem[CHH]{CHH} T.~D.~Cochran, S.~Harvey and P.~D.~Horn, \textit{Higher-order signature cocycles for subgroups of mapping class groups and homology cylinders}, Int. Math. Res. Not. IMRN, \textbf{14} (2012), 3311--3373.

  \bibitem[DHW]{DHW} K. Dekimpe, M. Hartl, and S. Wauters,
  \emph{A seven-term exact sequence for the cohomology of a group extension},
  J. Algebra, Volume 369, (2012), 70--95.

  \bibitem[E]{E} M. Entov,
	\textit{Quasi-morphisms and quasi-states in symplectic topology},
	The Proceedings of the International Congress of Mathematicians (Seoul, 2014).

\bibitem[EP03]{EP03} M. Entov and L. Polterovich,
\textit{Calabi quasimorphism and quantum homology},
Int.\ Math.\ Res.\ Not. \textbf{30} (2003), 1635--1676.

  \bibitem[EP06]{EP06} M.~Entov and L.~Polterovich,
  \textit{Quasi-states and symplectic intersections},
  Comment. Math. Helv., \textbf{81} (1) (2006), 75--99.

\bibitem[EP09]{EP09} M. Entov and L. Polterovich,
	\textit{Rigid subsets of symplectic manifolds},
	Compos.\ Math.\ \textbf{145} (2009), no.~3, 773--826.


\bibitem[EPP]{EPP} M. Entov, L. Polterovich and P. Py,
	\textit{On continuity of quasimorphisms for symplectic maps},
	in Perspectives in Analysis, Geometry, and Topology, eds.\ I. Itenberg, B. J\"oricke and M. Passare (Birkh\"auser/Springer, 2012), 169--197.

\bibitem[Fa]{Fa}
{A. Fathi},
\textit{Structure of the group of homeomorphisms preserving a good measure on a compact manifold},
{Ann. Sci. $\acute{\mathrm{E}}$cole Norm. Sup. (4) },
 \textbf{13}(1)  (1980), 45--93.

 \bibitem[Fr]{Fr}
 {R. Frigerio},
 \textit{Bounded Cohomology of Discrete Groups},
 {Mathematical Surveys and Monographs 227}, Amer. Math. Soc., (2017).


	
\bibitem[FOOO]{FOOO}
K. Fukaya, Y.-G. Oh, H. Ohta and K. Ono,
\textit{Spectral invariants with bulk, quasimorphisms and Lagrangian Floer theory},
Mem.\ Amer.\ Math.\ Soc. \textbf{260} (2019) no.~1254.

%  \bibitem[Ka16]{Ka16} M.~Kawasaki,  \emph{Relative quasimorphisms and stably unbounded norms on the group of symplectomorphisms of the Euclidean spaces},  J. Symplectic Geom., 14(2016), 297--304.


  \bibitem[GG]{GG} J-M. Gambaudo and E. Ghys,
  \textit{Commutators and diffeomorphisms of surfaces},
  Ergodic Theory Dynam. Systems, \textbf{24}(5) (2004), 1591--1617.

  \bibitem[GL]{GL} E. A. Gorin, V. Ya. Lin,
  \textit{Algebraic equations with continuous coefficients and some problems of the algebraic theory of braids},
  Math. USSR Sbornik, \textbf{7} No.4 (1969), 569--596

  \bibitem[Gr]{Gr} M. Gromov,
  \textit{Volume and bounded cohomology}, Inst. Hautes Études Sci. Publ.Math., (56):5–99 (1983), 1982.

  \bibitem[I]{I} T.~Ishida,\
  \textit{Quasi-morphisms on the group of area-preserving diffeomorphisms of the 2-disk via braid groups},
  Proc. Amer. Math. Soc. Ser. B, \textbf{1} (2014), 43--51.

  \bibitem[Ka17]{Ka17} M.~Kawasaki,
  \textit{Bavard's duality theorem on conjugation-invariant norms},
  Pacific J. Math. \textbf{288} (2017), 157--170.

  \bibitem[Ka18]{Ka18} M.~Kawasaki,
  \textit{Extension problem of subset-controlled quasimorphism},
  Proc. Amer. Math. Soc. Ser. B, \textbf{5} (2018), 1--5.

  \bibitem[KO]{KO19} M. Kawasaki and R. Orita,
	\textit{Disjoint superheavy subsets and fragmentation norms},
	J. Topol.\ Anal., Online Ready, doi:10.1142/S179352532050017X.

  \bibitem[KM]{KM} D, Kotschick and S. Morita,
  \textit{Characteristic classes of foliated surface bundles with area-preserving holonomy},
  J. Differential Geom. 75 (2007), no. 2, 273--302.

\bibitem[MT]{MT} K. ~Mann and B.~Tshishiku, \textit{Realization problems for diffeomorphism groups},
Breadth in contemporary topology, Proc. Sympos. Pure Math., \textbf{102} (2019), 131--156.


  \bibitem[MS]{MS}
{D. McDuff and D. Salamon},
\textit{Introduction to symplectic topology},
{Oxford Mathematical Monographs. The Clarendon Press, Oxford University Press, New York},
 (1998).


\bibitem[MVZ]{MVZ}
{A. Monzner, N. Vichery and F. Zapolsky},
	\textit{Partial quasimorphisms and quasistates on cotangent bundles, and symplectic homogenization},
	J. Mod.\ Dyn., \textbf{6} (2012) no.~2 205--249.

%\bibitem[Mo]{Mo}
%  S. Morita,
%   \textit{Characteristic classes of surface bundles}. Invent. Math. \textbf{90} (1987), 551–577.

 \bibitem[OM]{OM}
{Y. -G. Oh and S. M$\ddot{\mathrm{u}}$ller},
{\em  The group of Hamiltonian homeomorphisms and $C\sp 0$-symplectic topology},
{ J. Symplectic Geom. },
 \textbf{5} (2) (2007),  167--219.



  \bibitem[P]{P01}
{L. Polterovich},
\textit{The geometry of the group of symplectic diffeomorphisms},
{Lectures in Mathematics ETH Z$\ddot{\mathrm{u}}$rich. Birk$\ddot{\mathrm{a}}$user Verlag, Basel},
 (2001).

  \bibitem[PR]{PR}
{L. Polterovich and D. Rosen},
\textit{Function theory on symplectic manifolds},
{CRM Monograph Series, 34. American Mathematical Society, Providence, RI},
 (2014).


\bibitem[Py]{Py06}
{P. Py},
	\textit{Quasi-morphismes et invariant de Calabi},
	Ann.\ Sci.\ \'Ecole Norm.\ Sup. \textbf{4} (2006) no.~1 177--195.


\bibitem[R]{R}
{M. Rosenberg},
\textit{Py-Calabi quasi-morphisms and quasi-states on orientable surfaces of higher genus},
Isr.\ J. Math., \textbf{180} (2010) 163--188.

  \bibitem[Sh]{Sh} {A.~I.~Shtern}, \textit{Extension of pseudocharacters from normal subgroups}, Proc. Jangjeon Math. Soc., \textbf{18} (2015), 427--433.


\bibitem[vN]{vN}
{J. von Neumann},
\textit{Zur allgemeinen theorie des masses},
Fund. Math.
 \textbf{13} (1929) 73--116.


%Kȩdra, J. ;  Kotschick, D. ;  Morita, S.  Crossed flux homomorphisms and vanishing theorems for flux groups.
% Geom. Funct. Anal.  16  (2006),  no. 6, 1246--1273.


\end{thebibliography}
\end{document}